\documentclass[11pt]{article}
\usepackage[margin=1in]{geometry}
\usepackage[usenames,dvipsnames,svgnames,table]{xcolor} 

\usepackage{longtable}
\usepackage{hhline}
\usepackage{float}
\usepackage{enumitem}
\usepackage{booktabs}

\usepackage{setspace}
\usepackage{xspace}
\usepackage{enumitem} 
\usepackage{rotfloat} 
\usepackage{graphicx}
\usepackage{amssymb, amsmath, amsthm}
\usepackage{verbatim}
\usepackage{natbib}
\usepackage{authblk}
\usepackage{multirow}
\usepackage{subcaption} 
\usepackage{array}
\usepackage{hyperref}
\newcolumntype{L}{>{\centering\arraybackslash}m{0.1\linewidth}}

\newtheorem{theorem}{Theorem}[section]
\newtheorem{lemma}[theorem]{Lemma}
\newtheorem{proposition}[theorem]{Proposition}
\newtheorem{corollary}[theorem]{Corollary}

\newtheorem{definition}{Definition}
\newtheorem{assumption}[theorem]{Assumption}

\newcommand{\bbR}{\mathbb{R}}

\newcommand{\bbH}{\mathbb{H}}

\DeclareMathOperator{\arcosh}{arcosh}
\DeclareMathOperator{\tr}{tr}
\DeclareMathOperator{\vech}{vech}
\DeclareMathOperator{\mat}{mat}
\DeclareMathOperator{\argmax}{argmax}
\DeclareMathOperator{\argmin}{argmin}
\DeclareMathOperator{\KL}{KL}
\DeclareMathOperator{\Cov}{Cov}
\DeclareMathOperator{\diag}{diag}

\newcommand{\KLtext}{Kullback--Leibler\xspace}
\newcommand{\R}{\mathbb{R}}
\newcommand{\E}{\mathbb{E}}
\newcommand{\Prob}{\mathbb{P}}
\newcommand{\Hh}{\mathbb{H}}
\newcommand{\SPD}{\mathcal{S}^{++}}
\newcommand{\Sym}{\mathcal{S}}
\newcommand{\Id}{I}
\newcommand{\Exp}{\operatorname{Exp}}
\newcommand{\Log}{\operatorname{Log}}
\newcommand{\norm}[1]{\left\lVert #1\right\rVert}

\newcommand{\Loro}[2]{\left\langle #1,#2\right\rangle_{\;L}}
\newcommand{\shell}{\mathcal{K}}
\newcommand{\psiwrap}{\varphi}
\newcommand{\score}{\mathsf{s}}
\newcommand{\info}{\mathcal{I}}
\newcommand{\T}{\mathsf{T}}

\newcommand{\HajekLecam}{H\'ajek--Le Cam\xspace}

\begin{document}
	
\title{Profile Likelihood Inference for Anisotropic Hyperbolic Wrapped Normal Models on Hyperbolic Space}
\author[1,2]{Kisung You}
\affil[1]{Department of Mathematics, Baruch College}
\affil[2]{Department of Mathematics, The Graduate Center, City University of New York}
\date{}

\maketitle
\begin{abstract}
We study likelihood-based inference for the anisotropic hyperbolic wrapped normal distribution on standard hyperbolic space. The model has a manifold-valued location parameter and a full positive definite covariance matrix in tangent coordinates. For independent observations from this family, we analyze the profile maximum likelihood estimator obtained by optimizing the likelihood over the location after profiling out the covariance. To guarantee finite-sample existence, we formulate the estimator on a covariance shell that bounds eigenvalues away from zero and infinity. We prove that this constrained likelihood is well posed, that the anisotropic wrapped normal family is identifiable, and that the estimator is strongly consistent when the true covariance lies in the interior of the shell. In global normal coordinates for the location and log-covariance coordinates for the nuisance parameter, we establish joint asymptotic normality and derive efficient profile inference for the location parameter through the Schur-complement information. We further prove local asymptotic normality of the experiment and obtain the \HajekLecam local asymptotic minimax lower bound under squared geodesic loss. The profile estimator attains this bound for truncated squared loss, and for ordinary squared loss under a uniform-integrability condition. We also give an explicit computational form of the estimator based on spectral clipping of the empirical tangent covariance, and present a Monte Carlo calibration study showing that the finite-sample scaled location risk and Wald coverage agree with the asymptotic theory.
\end{abstract}

\section{Introduction}

Hyperbolic geometry has become a standard ambient geometry for data with latent tree-like or hierarchical structure, both in geometric statistics and representation learning \citep{nickel_2017_PoincareEmbeddingsLearning, nickel_2018_LearningContinuousHierarchies}. A central probabilistic construction in that literature is the hyperbolic wrapped normal (HWN) distribution introduced by \citet{nagano_2019_WrappedNormalDistribution}. Their model is defined by sampling a Gaussian vector in the tangent space at the hyperbolic origin, parallel-transporting it to a desired base point, and projecting it onto the manifold. A key virtue of the construction is that the density is available in closed form, because in the Lorentz model the Jacobian of the exponential map is explicitly available.

Despite the importance of the model, inferential theory for its parameters appears to be largely absent from the statistical literature. The gap is especially visible in the anisotropic case, where the covariance parameter is a full positive definite matrix. The anisotropic model is substantially richer than the isotropic one \citep{said_2018_GaussianDistributionsRiemannian, you_2026_FiniteMixtureModeling}, but also more delicate. That is, the location parameter lives on a noncompact manifold, the likelihood includes a curvature-induced Jacobian term, and the unrestricted conditional covariance maximization can fail to have a finite maximizer when the empirical tangent covariance becomes singular.

This paper develops a rigorous likelihood-based theory for the anisotropic hyperbolic wrapped normal family on the standard hyperbolic space $\bbH^d$, where the estimator is developed by the profile likelihood framework \citep{box_1964_AnalysisTransformations, cox_1970_AnalysisBinaryData}. The location parameter is estimated by maximizing the likelihood after profiling out the covariance in tangent coordinates. The correct minimax conclusion for this route is not exact finite-sample minimaxity on the whole manifold. Rather, it is local asymptotic minimaxity in the sense of H\'ajek and Le Cam. This is the natural decision-theoretic endpoint of regular likelihood theory on curved noncompact parameter spaces \citep{hajek_1972_LocalAsymptoticMinimax, vandervaart_1998_AsymptoticStatistics}. Our contributions are as follows.

\begin{enumerate}
\item We formulate a shell-constrained anisotropic likelihood on
\[
\Theta=\Hh^d\times \shell,
\qquad
\shell=\{\Sigma\in\SPD(d):\lambda_-\Id\preceq\Sigma\preceq\lambda_+\Id\},
\]
with $0<\lambda_-<\lambda_+$ fixed and the true covariance in the interior. This resolves the finite-sample non-attainment issue for the unrestricted covariance likelihood while preserving the practical closed form of the covariance profile whenever the empirical tangent covariance lies in the shell.
\item We prove identifiability of the anisotropic HWN family and existence of the shell-constrained joint and profile MLEs.
\item We prove strong consistency by combining a coercivity argument in the location parameter, a compact-uniform law of large numbers, and \KLtext identification in a global normal coordinate chart.
\item In local Euclidean coordinates $(\alpha,\beta)$, where $\alpha=\Log_{\mu_0}(\mu)$ and $\beta=\vech(\log\Sigma)$, we prove joint asymptotic normality of the MLE and efficient asymptotic normality of the profiled location estimator. The efficient covariance for location is the inverse Schur complement of the nuisance block of the Fisher information.
\item We establish local asymptotic normality (LAN) of the anisotropic HWN experiment and deduce the \HajekLecam local asymptotic minimax lower bound. The profile MLE attains the corresponding bound for bounded and truncated squared geodesic losses; the untruncated squared-loss conclusion is obtained under an explicit local uniform-integrability condition.
\item We translate the estimator into a concrete computational recipe. In particular, we derive the closed spectral form of the shell-constrained covariance profile, explain how the outer optimization can be performed in global normal coordinates, and include a Monte Carlo study assessing finite-sample calibration of the asymptotic theory.
\end{enumerate}

\paragraph{Related literature.}
General asymptotic theory for M-estimators and maximum likelihood estimators follows the classical lines of \citet{wald_1949_NoteConsistencyMaximum}, \citet{hoadley_1971_AsymptoticPropertiesMaximum}, and \citet{vandervaart_1998_AsymptoticStatistics}. For local asymptotic minimax theory we rely on \citet{hajek_1972_LocalAsymptoticMinimax} and the modern presentation in \citet{vandervaart_1998_AsymptoticStatistics}. Background on the geometry of Hadamard manifolds and normal coordinates can be found in \citet{petersen_2006_RiemannianGeometry}, and the matrix-logarithm chart for positive definite matrices is standard \citep{bhatia_2009_PositiveDefiniteMatrices}. For optimization on manifolds we will use the general framework of \citet{absil_2008_OptimizationAlgorithmsMatrix}. We emphasize that our problem is different from intrinsic Frechet mean estimation on manifolds, treated for example by \citet{bhattacharya_2003_LargeSampleTheory}. Here the target is a fully parametric anisotropic wrapped family rather than a nonparametric notion of center. We also note that our LAN and local minimax arguments fit into the general asymptotic-statistical paradigm developed by \citet{lecam_2000_AsymptoticsStatistics}, although the concrete verification here is carried out directly for the anisotropic HWN model.

\medskip
The rest of the paper is organized as follows. Section~\ref{sec:model} introduces the anisotropic HWN model, its density, the sample likelihood, and the shell-constrained profile estimator. Section~\ref{sec:prelim} develops the analytic preliminaries needed for likelihood theory, including moment bounds, derivative envelopes, score identities, and a compact-uniform law of large numbers. Section~\ref{sec:existence_consistency} proves existence, identifiability, strong consistency, and asymptotic inactivity of the covariance shell. Section~\ref{sec:normality} establishes joint asymptotic normality and efficient profile inference for the location parameter. Section~\ref{sec:lan} proves local asymptotic normality and derives the local asymptotic minimax consequences under squared geodesic loss. Section~\ref{sec:computation} describes the computational form of the estimator, including the spectral covariance profile and the global coordinate used for location optimization. Section~\ref{sec:numerical} presents a Monte Carlo calibration study. The proofs are collected in the Appendix.


\section{The anisotropic hyperbolic wrapped normal model}
\label{sec:model}

\subsection{Hyperbolic space in the Lorentz model}

Let $d\ge 2$. We write $\Hh^d$ for the $d$-dimensional hyperbolic space in the Lorentz model:
\[
\Hh^d=\{x\in\R^{d+1}:\Loro{x}{x}=-1,\ x_0>0\},
\qquad
\Loro{x}{y}=-x_0y_0+\sum_{j=1}^d x_jy_j .
\]
The induced Riemannian metric on each tangent space $T_x\Hh^d$ is the restriction of the Lorentz bilinear form. The hyperbolic distance is
\[
\rho(x,y)=\arcosh\bigl(-\Loro{x}{y}\bigr),
\qquad x,y\in\Hh^d.
\]
We denote by $o=(1,0,\dots,0)$ the origin. The exponential and logarithm maps are explicit:
\begin{align}
\Exp_\mu(v)
&=\cosh(\norm{v})\,\mu+\sinh(\norm{v})\frac{v}{\norm{v}},
\qquad v\in T_\mu\Hh^d, \label{eq:exp-lorentz}\\
\Log_\mu(x)
&=\frac{\rho(\mu,x)}{\sqrt{(-\Loro{\mu}{x})^2-1}}\bigl(x+\Loro{\mu}{x}\,\mu\bigr),
\qquad x\in\Hh^d,\ x\neq \mu , \label{eq:log-lorentz}
\end{align}
with the convention $\Log_\mu(\mu)=0$. Because $\Hh^d$ is a Hadamard manifold, $\Exp_\mu:T_\mu\Hh^d\to\Hh^d$ is a global diffeomorphism for every $\mu$ \citep{petersen_2006_RiemannianGeometry}. Parallel transport along the unique geodesic from $x$ to $y$ is denoted by $PT_{x\to y}:T_x\Hh^d\to T_y\Hh^d$, whose explicit formulas are given in \citet{nagano_2019_WrappedNormalDistribution}.

We identify $T_o\Hh^d$ with $\R^d$ by $(0,v)\leftrightarrow v$. For each $\mu\in\Hh^d$, define
\[
\T_\mu(x):=PT_{o\to\mu}^{-1}\bigl(\Log_\mu(x)\bigr)\in\R^d,
\qquad
r_\mu(x):=\norm{\T_\mu(x)}=\rho(\mu,x).
\]
The map $x\mapsto \T_\mu(x)$ is a global coordinate map from $\Hh^d$ onto $\R^d$ centered at $\mu$.

\subsection{Definition and density}

\begin{definition}[Anisotropic hyperbolic wrapped normal]
Let $\mu\in\Hh^d$ and $\Sigma\in\SPD(d)$. A random variable $X$ has the anisotropic hyperbolic wrapped normal distribution, denoted by
\[
X\sim \mathrm{HWN}_d(\mu,\Sigma),
\]
if there exists $Z\sim N_d(0,\Sigma)$ such that
\[
X=\Exp_\mu\bigl(PT_{o\to\mu}Z\bigr).
\]
\end{definition}
This is exactly the wrapped construction of \citet{nagano_2019_WrappedNormalDistribution}, specialized here to statistical inference on the parameter pair $(\mu,\Sigma)$.

\begin{proposition}[Density formula]
\label{prop:density}
The distribution $\mathrm{HWN}_d(\mu,\Sigma)$ has density with respect to the hyperbolic Riemannian volume measure
\begin{equation}
\label{eq:density}
p_{\mu,\Sigma}(x)
=(2\pi)^{-d/2}|\Sigma|^{-1/2}
\exp\;\Bigl(-\tfrac12\T_\mu(x)^\top\Sigma^{-1}\T_\mu(x)\Bigr)
\Bigl(\frac{r_\mu(x)}{\sinh r_\mu(x)}\Bigr)^{d-1},
\qquad x\in\Hh^d,
\end{equation}
where the last factor is interpreted as $1$ at $x=\mu$.
\end{proposition}

It is convenient to remove the removable singularity in the Jacobian term by introducing the smooth radial function
\begin{equation}
\label{eq:phi-def}
\psiwrap(v):=
\begin{cases}
\log\left(\dfrac{\sinh \norm{v}}{\norm{v}}\right), & v\neq 0,\\[0.75em]
0, & v=0.
\end{cases}
\end{equation}
Because
\[
\log\left(\frac{\sinh r}{r}\right)=\frac{r^2}{6}-\frac{r^4}{180}+O(r^6),
\qquad r\to 0,
\]
the map $\psiwrap$ is $C^\infty$ on all of $\R^d$.

For a single observation $x\in\Hh^d$, dropping the additive constant $-(d/2)\log(2\pi)$, we write the log-density as
\begin{equation}
\label{eq:single-loglik}
m_{\mu,\Sigma}(x)
=
-\frac12\log\det\Sigma
-\frac12\T_\mu(x)^\top\Sigma^{-1}\T_\mu(x)
-(d-1)\psiwrap\bigl(\T_\mu(x)\bigr).
\end{equation}

\subsection{Likelihood, covariance profiling, and the shell-constrained estimator}

Fix observations $X_1,\dots,X_n\in\Hh^d$. The sample log-likelihood is
\begin{equation}
\label{eq:sample-loglik}
\ell_n(\mu,\Sigma)=\sum_{i=1}^n m_{\mu,\Sigma}(X_i).
\end{equation}
For a fixed $\mu$, define the empirical tangent covariance
\begin{equation}
\label{eq:Sn}
S_n(\mu):=\frac1n\sum_{i=1}^n \T_\mu(X_i)\T_\mu(X_i)^\top.
\end{equation}
Then \eqref{eq:sample-loglik} can be rewritten as
\begin{equation}
\label{eq:sample-loglik2}
\ell_n(\mu,\Sigma)
=
-\frac n2\log\det\Sigma
-\frac n2\tr\bigl(S_n(\mu)\Sigma^{-1}\bigr)
-(d-1)\sum_{i=1}^n \psiwrap\bigl(\T_\mu(X_i)\bigr).
\end{equation}
If $S_n(\mu)$ is positive definite and the covariance parameter is unrestricted over $\SPD(d)$, the conditional maximizer is $\Sigma=S_n(\mu)$, exactly as in the Euclidean Gaussian model. However, if $S_n(\mu)$ is singular, no finite unrestricted maximizer exists and the likelihood can be increased indefinitely by shrinking the covariance in directions orthogonal to the sample span. To avoid this pathology, fix constants $0<\lambda_-<\lambda_+$ and define the compact shell
\begin{equation}
\label{eq:shell}
\shell:=\{\Sigma\in\SPD(d):\lambda_-\Id\preceq\Sigma\preceq\lambda_+\Id\}.
\end{equation}
We assume throughout that the true covariance $\Sigma_0$ lies in $\mathrm{int}(\shell)$.

For each $\mu$, define the shell-constrained covariance profile
\begin{equation}
\label{eq:sigma-profile}
\widetilde\Sigma_n(\mu)\in
\argmin_{\Sigma\in\shell}
\Bigl\{\log\det\Sigma+\tr\bigl(S_n(\mu)\Sigma^{-1}\bigr)\Bigr\}.
\end{equation}
Then any shell-constrained profile MLE is a measurable maximizer of
\begin{equation}
\label{eq:mu-profile-est}
\hat\mu_n\in \argmax_{\mu\in\Hh^d}\ell_n\bigl(\mu,\widetilde\Sigma_n(\mu)\bigr),
\qquad
\hat\Sigma_n:=\widetilde\Sigma_n(\hat\mu_n).
\end{equation}
The next proposition constructs the basic covariance-profile algebra.

\begin{proposition}[Conditional covariance optimization]
\label{prop:conditionalSigma}
For every symmetric positive semidefinite matrix $S$, the map
\[
\Sigma\mapsto f_S(\Sigma):=\log\det\Sigma+\tr(S\Sigma^{-1})
\]
is continuous on $\shell$ and hence attains its minimum on $\shell$. If $S\in\SPD(d)$, then the unique unconstrained minimizer on $\SPD(d)$ is $S$. Consequently, if $S\in \mathrm{int}(\shell)$, then the unique constrained minimizer on $\shell$ is also $S$.
\end{proposition}

In the practically important event $S_n(\mu)\in \mathrm{int}(\shell)$, the profile objective reduces to
\begin{equation}
\label{eq:practical-profile-criterion}
\mu\mapsto
-\frac n2\log\det S_n(\mu)
-(d-1)\sum_{i=1}^n \psiwrap\bigl(\T_\mu(X_i)\bigr),
\end{equation}
up to the additive constant $-nd/2$. This is the objective one would optimize in practice when the empirical tangent covariance is positive definite.

\section{Analytic preliminaries}
\label{sec:prelim}

The proofs throughout this paper use a global Euclidean chart centered at the true location $\mu_0$ and a Euclidean chart for $\SPD(d)$ based on the matrix logarithm. The location chart is
\begin{equation}
\label{eq:alpha-chart}
\alpha=\Log_{\mu_0}(\mu)\in T_{\mu_0}\Hh^d\cong \R^d,
\qquad
\mu(\alpha)=\Exp_{\mu_0}(\alpha),
\end{equation}
which is global because $\Hh^d$ is Hadamard \citep{lee_1997_RiemannianManifoldsIntroduction}. For the covariance, let $m=d(d+1)/2$ and define
\begin{equation}
\label{eq:beta-chart}
\beta=\vech(\log\Sigma)\in\R^m,
\qquad
\Sigma(\beta)=\exp(\mat(\beta)).
\end{equation}
The matrix logarithm is a smooth global diffeomorphism from $\SPD(d)$ onto the space of symmetric matrices \citep{bhatia_2009_PositiveDefiniteMatrices}. We write
\[
\vartheta=(\alpha,\beta)\in\R^p,
\qquad
p=d+m,
\]
with true value $\vartheta_0=(0,\beta_0)$.

\subsection{Tail moments under the true model}

We first record basic moment and tail properties of the anisotropic HWN distribution. These bounds ensure the integrability conditions needed for the likelihood-based analysis, which are established by the following lemma.

\begin{lemma}[Moment and exponential tail bounds]\label{lem:moments}
Let $X\sim \mathrm{HWN}_d(\mu_0,\Sigma_0)$. Then:
\begin{enumerate}[label=(\alph*)]
\item for every $q\ge 1$, $\E[\rho(o,X)^q]<\infty$;
\item for every $a>0$, $\E[\exp\{a\rho(o,X)\}]<\infty$.
\end{enumerate}
Moreover, these bounds are uniform over parameter sets of the form $K_\mu\times K_\Sigma$, where $K_\mu\subset\Hh^d$ is compact and $K_\Sigma\subset\SPD(d)$ has eigenvalues bounded above by a finite constant.
\end{lemma}

\subsection{Smoothness and derivative envelopes}

The next result is the technical regularity statement behind all asymptotic arguments. It shows that the anisotropic HWN log-likelihood is a smooth finite-dimensional criterion whose derivatives are dominated on compact parameter sets by integrable envelopes under the true model.

\begin{proposition}[Smoothness and derivative envelopes]
\label{prop:envelope}
Let $K_\mu\subset\Hh^d$ be compact and let $B\subset\R^m$ be compact. For $j=0,1,2,3$ there exist constants $C_j,q_j<\infty$ such that
\begin{equation}
\label{eq:envelope}
\sup_{\mu\in K_\mu,\ \beta\in B}
\bigl\|\partial_{(\alpha,\beta)}^j m_{\mu,\Sigma(\beta)}(x)\bigr\|
\le
C_j\bigl(1+\rho(o,x)^{q_j}\bigr),
\qquad x\in\Hh^d,
\end{equation}
where derivatives are taken in the Euclidean chart $(\alpha,\beta)$ with $\mu=\mu(\alpha)$.
In particular, for every compact $K\subset\R^p$,
\[
\sup_{\vartheta\in K}\bigl\|\partial_\vartheta^j m_\vartheta(X)\bigr\|
\]
has finite expectation under $P_{\vartheta_0}$ for $j=0,1,2,3$, and finite $(2+\delta)$th moment for $j=1$ for every $\delta>0$.
\end{proposition}

Proposition~\ref{prop:envelope} is the reason a local Euclidean likelihood analysis works. Once the model is written in the global chart $(\alpha,\beta)$, all standard maximum-likelihood arguments become available. In particular, compact-uniform laws of large numbers, differentiation under the integral sign, score central limit theorems, and third-order Taylor expansions can be readily justified by domination.

\subsection{Score, information, and a compact-uniform law of large numbers}

For $\vartheta=(\alpha,\beta)\in\R^p$, write $m_\vartheta(x)$ for the log-density \eqref{eq:single-loglik} expressed in the chart $(\alpha,\beta)$. Define the score and information matrix by
\[
\score_\vartheta(x):=\partial_\vartheta m_\vartheta(x),
\qquad
\info(\vartheta):=\E_\vartheta[\score_\vartheta(X)\score_\vartheta(X)^\top].
\]
We next record basic properties of the score function and a compact-uniform law of large numbers. The former ensures that the score has mean zero and identifies the Fisher information, while the latter provides uniform convergence of the empirical criterion and its derivatives on compact parameter sets.

\begin{lemma}[Score mean zero and information identity]
\label{lem:scoremean}
For every $\vartheta$,
\[
\E_\vartheta[\score_\vartheta(X)]=0,
\qquad
\info(\vartheta)= -\,\E_\vartheta[\partial_\vartheta^2 m_\vartheta(X)].
\]
\end{lemma}

\begin{lemma}[Compact-uniform law of large numbers]
\label{lem:ulln}
Let $K\subset\R^p$ be compact and define
\[
M_n(\vartheta):=\frac1n\sum_{i=1}^n m_\vartheta(X_i),
\qquad
M(\vartheta):=\E_{\vartheta_0}[m_\vartheta(X)].
\]
Then
\[
\sup_{\vartheta\in K}\bigl|M_n(\vartheta)-M(\vartheta)\bigr|\to 0
\qquad\text{a.s.}
\]
Moreover, the same conclusion holds if $m_\vartheta$ is replaced by any first or second partial derivative of $m_\vartheta$.
\end{lemma}

\section{Existence, identifiability, and strong consistency}\label{sec:existence_consistency}

We now establish the basic structural properties of the likelihood framework. In particular, we show that the shell-constrained likelihood admits a maximizer, that the anisotropic HWN family is identifiable, and that the resulting estimator is strongly consistent under standard sampling assumptions. These results provide the foundation for the subsequent asymptotic analysis.

\begin{theorem}[Existence of the shell-constrained MLE]
\label{thm:existence}
For every sample $(X_1,\dots,X_n)\in(\Hh^d)^n$, the shell-constrained joint likelihood
\[
(\mu,\Sigma)\mapsto \ell_n(\mu,\Sigma)
\]
admits at least one maximizer on $\Theta=\Hh^d\times\shell$. Consequently, the shell-constrained profile MLE \eqref{eq:mu-profile-est} exists.
\end{theorem}

The existence result ensures that the constrained optimization problem is well posed for every sample. We next verify that the underlying statistical model is identifiable, which guarantees that the population likelihood has a unique maximizer.

\begin{theorem}[Identifiability]
\label{thm:identifiability}
The anisotropic HWN family $\{\mathrm{HWN}_d(\mu,\Sigma):(\mu,\Sigma)\in \Hh^d\times \SPD(d)\}$ is identifiable. That is, if
$\mathrm{HWN}_d(\mu,\Sigma)=\mathrm{HWN}_d(\mu',\Sigma')$, then $\mu=\mu'$ and $\Sigma=\Sigma'$.
\end{theorem}

To study asymptotic properties, we impose standard regularity conditions. The first ensures that the true parameter lies in the interior of the constraint set, so that the shell does not bind asymptotically. The second is the usual nondegeneracy condition on the Fisher information.

\begin{assumption}[Interior truth and nonsingular information]
\label{ass:interior-info}
The true parameter $(\mu_0,\Sigma_0)$ belongs to $\Hh^d\times \mathrm{int}(\shell)$, and the Fisher information matrix $\info(\vartheta_0)$ is positive definite.
\end{assumption}

The interior condition is required only because the shell is a technical device. The shell itself may be chosen arbitrarily large provided it contains the true covariance in its interior. The positive-definiteness of the information matrix is the standard nondegeneracy condition needed for asymptotic normality and LAN.

We now turn to consistency of the estimator. The proof follows a classical likelihood argument based on Kullback--Leibler identification, localization of the maximizer on a compact set, and a compact-uniform law of large numbers.

\begin{theorem}[Strong consistency]
\label{thm:consistency}
Assume that $X_1,X_2,\dots$ are i.i.d.\ from $\mathrm{HWN}_d(\mu_0,\Sigma_0)$ with $(\mu_0,\Sigma_0)\in\Hh^d\times \mathrm{int}(\shell)$. Let $(\hat\mu_n,\hat\Sigma_n)$ be any measurable shell-constrained maximizer of the likelihood. Then
\[
\hat\mu_n\to \mu_0,
\qquad
\hat\Sigma_n\to \Sigma_0
\qquad\text{almost surely}.
\]
Equivalently, in the Euclidean chart,
\[
\hat\vartheta_n\to \vartheta_0
\qquad\text{a.s.}
\]
\end{theorem}

The next proposition shows that the shell constraint becomes inactive asymptotically at the estimator. This provides the link between the constrained formulation used for existence and the unconstrained profile estimator employed in practice.

\begin{proposition}[Asymptotic inactivity of the shell]
\label{prop:inactive-shell}
Under the assumptions of Theorem~\ref{thm:consistency},
\[
S_n(\hat\mu_n)\to \Sigma_0
\qquad\text{a.s.}
\]
Consequently,
\[
\Prob\bigl(S_n(\hat\mu_n)\in \mathrm{int}(\shell)\bigr)\to 1,
\]
and on that event
\[
\hat\Sigma_n=S_n(\hat\mu_n).
\]
\end{proposition}

Taken together, these results show that the shell-constrained likelihood provides a well-posed and statistically consistent framework for inference in the anisotropic HWN model. In the next section, we build on this foundation to derive the asymptotic distribution of the estimator and establish efficiency properties via local asymptotic normality.

\section{Asymptotic normality and efficient profile inference}
\label{sec:normality}

We now develop the second-order asymptotic theory for the shell-constrained maximum likelihood estimator. Building on the consistency results of the previous section, we derive an expansion for the estimator, establish joint asymptotic normality, and characterize the efficient limit distribution of the profiled location parameter. All statements are made in the Euclidean chart $\vartheta=(\alpha,\beta)$ introduced in \eqref{eq:alpha-chart}--\eqref{eq:beta-chart}. Throughout this section, write
\[
\score_{\vartheta}(x)=\partial_\vartheta m_\vartheta(x),
\qquad
\dot\score_{\vartheta}(x)=\partial_\vartheta^2 m_\vartheta(x).
\]

We begin with a standard likelihood expansion showing that the estimator admits an asymptotically linear representation. This yields joint asymptotic normality for the full parameter.

\begin{theorem}[Joint asymptotic normality of the constrained MLE]
\label{thm:joint-asym-normal}
Assume that $X_1,X_2,\dots$ are i.i.d. from $\mathrm{HWN}_d(\mu_0,\Sigma_0)$ and that Assumption~\ref{ass:interior-info} holds. Then
\begin{equation}
\label{eq:joint-linear}
\sqrt n\,(\hat\vartheta_n-\vartheta_0)
=
\info(\vartheta_0)^{-1}\frac1{\sqrt n}\sum_{i=1}^n \score_{\vartheta_0}(X_i)+o_p(1),
\end{equation}
and hence
\begin{equation}
\label{eq:joint-normal}
\sqrt n\,(\hat\vartheta_n-\vartheta_0)\ \Rightarrow\
N\bigl(0,\info(\vartheta_0)^{-1}\bigr).
\end{equation}
\end{theorem}

To study the effect of the covariance parameter as a nuisance, we decompose the Fisher information matrix according to the partition $\vartheta=(\alpha,\beta)$:
\begin{equation}
\label{eq:block-info}
\info(\vartheta_0)
=
\begin{pmatrix}
\info_{\alpha\alpha} & \info_{\alpha\beta}\\[0.25em]
\info_{\beta\alpha} & \info_{\beta\beta}
\end{pmatrix}.
\end{equation}
Because $\info(\vartheta_0)$ is positive definite, the nuisance block $\info_{\beta\beta}$ is invertible. Define the efficient information for the location coordinate after profiling out the covariance nuisance:
\begin{equation}
\label{eq:schur}
\info_{\alpha\alpha\cdot\beta}
:=
\info_{\alpha\alpha}
-
\info_{\alpha\beta}\info_{\beta\beta}^{-1}\info_{\beta\alpha}.
\end{equation}
Also define the efficient score
\begin{equation}
\label{eq:efficient-score}
\tilde\score_\alpha(x)
:=
\score_\alpha(x)-\info_{\alpha\beta}\info_{\beta\beta}^{-1}\score_\beta(x).
\end{equation}
We now extract the asymptotic behavior of the profiled location estimator. The resulting limit distribution is governed by the Schur complement of the nuisance information block and corresponds to the efficient information for $\alpha$.

\begin{corollary}[Efficient asymptotic normality of the profiled location estimator]
\label{cor:profile-location}
Under the assumptions of Theorem~\ref{thm:joint-asym-normal},
\begin{equation}
\label{eq:profile-linear}
\sqrt n\,\hat\alpha_n
=
\info_{\alpha\alpha\cdot\beta}^{-1}
\frac1{\sqrt n}\sum_{i=1}^n \tilde\score_\alpha(X_i)
+o_p(1),
\end{equation}
and
\begin{equation}
\label{eq:profile-normal}
\sqrt n\,\hat\alpha_n \Rightarrow
N\bigl(0,\info_{\alpha\alpha\cdot\beta}^{-1}\bigr),
\qquad
\hat\alpha_n:=\Log_{\mu_0}(\hat\mu_n)\in \R^d.
\end{equation}
\end{corollary}
This shows that profiling out the covariance parameter leads to an asymptotically efficient estimator for the location parameter in the presence of nuisance covariance.

We close this section with a brief comment on the full-parameter interpretation. Although our main decision-theoretic conclusion concerns the location parameter under nuisance covariance, Theorem~\ref{thm:joint-asym-normal} already implies asymptotic efficiency of the full joint estimator $(\hat\alpha_n,\hat\beta_n)$ for smooth local quadratic losses on the whole parameter. The profile focus is motivated by the fact that the covariance is nuisance when estimating location.

\section{Local asymptotic normality and local asymptotic minimaxity}
\label{sec:lan}

We next turn from large-sample distribution theory to the local decision-theoretic implications of the likelihood expansion. The goal is to show that, after localizing the anisotropic HWN experiment around the true parameter, the model converges to a Gaussian shift experiment with Fisher information. This LAN structure then yields the local asymptotic minimax lower bound for estimating the location parameter under squared geodesic loss.

\subsection{The LAN expansion}

For $h\in\R^p$, define the local parameter perturbation
\[
\vartheta_{n,h}:=\vartheta_0+\frac{h}{\sqrt n}.
\]
Because $\Sigma_0$ lies in the interior of the covariance shell, for every compact $C\subset\R^p$ the covariance component of $\vartheta_{n,h}$ also belongs to the interior of the shell for all sufficiently large $n$, uniformly over $h\in C$. Define the log-likelihood ratio
\[
\Lambda_n(h)
:=
\log\frac{dP_{\vartheta_{n,h}}^n}{dP_{\vartheta_0}^n}
=
\sum_{i=1}^n \Bigl(m_{\vartheta_{n,h}}(X_i)-m_{\vartheta_0}(X_i)\Bigr).
\]
We also define the central sequence
\[
\Delta_n:=\frac1{\sqrt n}\sum_{i=1}^n \score_{\vartheta_0}(X_i).
\]
The following result is the key bridge between the likelihood expansion of the previous section and the minimax analysis below.

\begin{theorem}[Local asymptotic normality]
\label{thm:lan}
Assume that the experiment is generated by the anisotropic HWN family and that Assumption~\ref{ass:interior-info} holds. Then, uniformly for $h$ ranging over compact subsets of $\R^p$,
\begin{equation}
\label{eq:lan}
\Lambda_n(h)
=
h^\top \Delta_n
-\frac12 h^\top \info(\vartheta_0) h
+o_{P_{\vartheta_0}}(1),
\end{equation}
and
\[
\Delta_n\Rightarrow N\bigl(0,\info(\vartheta_0)\bigr)
\qquad\text{under }P_{\vartheta_0}^n.
\]
\end{theorem}
The LAN expansion is formulated in Euclidean local coordinates, whereas the natural loss for the original problem is geodesic loss on \(\Hh^d\). We therefore record the local equivalence between squared geodesic distance and squared Euclidean distance in normal coordinates.

\subsection{Local geodesic loss and Euclidean normal coordinates}

Because the parameter of interest lives on the manifold $\Hh^d$, the relevant loss is squared geodesic loss:
\[
L(\hat\mu,\mu):=\rho(\hat\mu,\mu)^2.
\]
To connect this loss to LAN theory, we need the local equivalence between geodesic distance and Euclidean distance in normal coordinates.

\begin{lemma}[Local equivalence of squared geodesic and Euclidean loss]
\label{lem:loss-equivalence}
Let $u_n,v_n\in T_{\mu_0}\Hh^d$ satisfy $\sqrt n\,u_n$ and $\sqrt n\,v_n$ bounded. Then
\[
n\,\rho\bigl(\Exp_{\mu_0}(u_n),\Exp_{\mu_0}(v_n)\bigr)^2
-
\norm{\sqrt n\,(u_n-v_n)}^2
\to 0.
\]
More generally, the convergence is uniform whenever $(\sqrt n\,u_n,\sqrt n\,v_n)$ ranges over a fixed compact subset of $T_{\mu_0}\Hh^d\times T_{\mu_0}\Hh^d$.
\end{lemma}
With this local loss equivalence in hand, we can translate the Gaussian-shift lower bound back to the original manifold-valued estimation problem.

\subsection{The \HajekLecam lower bound}

Write $h=(h_\alpha,h_\beta)$ according to the location/nuisance decomposition. Define the local alternatives
\[
\mu_{n,h}:=\Exp_{\mu_0}(h_\alpha/\sqrt n),
\qquad
\Sigma_{n,h}:=\exp\bigl(\mat(\beta_0+h_\beta/\sqrt n)\bigr),
\qquad
P_{n,h}:=P_{(\mu_{n,h},\Sigma_{n,h})}^n .
\]
For an estimator $\tilde\mu_n$, write
\[
\tilde a_n:=\sqrt n\,\Log_{\mu_0}(\tilde\mu_n)\in\R^d .
\]
This local coordinate is globally defined because $\Hh^d$ is Hadamard.

The proof of the lower bound uses two facts. First, the LAN experiment with nuisance covariance has efficient information for the location coordinate equal to
\[
\info_{\alpha\alpha\cdot\beta}
=
\info_{\alpha\alpha}-\info_{\alpha\beta}\info_{\beta\beta}^{-1}\info_{\beta\alpha}.
\]
Second, in a Hadamard manifold with nonpositive sectional curvature the exponential map is distance nondecreasing: for all $u,v\in T_{\mu_0}\Hh^d$,
\begin{equation}
\label{eq:exp-nondecreasing}
\rho\bigl(\Exp_{\mu_0}u,\Exp_{\mu_0}v\bigr)\ge \|u-v\|.
\end{equation}
This follows from the Cartan--Hadamard theorem and the Rauch comparison theorem \citep{docarmo_1992_RiemannianGeometry}. In  hyperbolic space it can also be checked directly from the hyperbolic law of cosines.

The lower bound is stated first for truncated squared geodesic loss, which avoids unnecessary integrability assumptions on arbitrary estimator sequences. The ordinary squared-loss bound then follows by monotone convergence.
\begin{theorem}[Local asymptotic minimax lower bound]
\label{thm:lower-bound}
Assume that the local experiments are generated by the anisotropic HWN family and that Assumption~\ref{ass:interior-info} holds. Let $\tilde\mu_n$ be any sequence of estimators with values in $\Hh^d$. If
$Z\sim N(0,\info_{\alpha\alpha\cdot\beta}^{-1})$, then for every $A<\infty$,
\begin{equation}
\label{eq:minimax-lower-trunc}
\lim_{c\to\infty}\ \liminf_{n\to\infty}\
\sup_{\|h\|\le c}
\E_{n,h}\Bigl[A\wedge n\rho\bigl(\tilde\mu_n,\mu_{n,h}\bigr)^2\Bigr]
\ge
\E\bigl[A\wedge \|Z\|^2\bigr].
\end{equation}
Consequently,
\begin{equation}
\label{eq:minimax-lower}
\lim_{c\to\infty}\ \liminf_{n\to\infty}\
\sup_{\|h\|\le c}
\E_{n,h}\Bigl[n\rho\bigl(\tilde\mu_n,\mu_{n,h}\bigr)^2\Bigr]
\ge
\tr\bigl(\info_{\alpha\alpha\cdot\beta}^{-1}\bigr).
\end{equation}
\end{theorem}
The preceding theorem applies to every estimator and gives the benchmark risk. We now show that the profiled location MLE attains this benchmark. For bounded and truncated losses this follows from regular efficiency. For unbounded squared loss, a uniform-integrability condition is needed to pass from convergence in distribution to convergence of risks.

\subsection{Attainment by the profile MLE}

The lower bound above is unconditional. Attainment by the profile MLE is immediate for bounded and truncated losses from regular efficiency. For ordinary unbounded squared loss, one additional moment condition is needed; without it, convergence in distribution does not imply convergence of risks.

\begin{assumption}[Local uniform integrability of squared location risk]
\label{ass:ui}
For every $c<\infty$, the random variables
\[
\left\{n\rho(\hat\mu_n,\mu_{n,h})^2:
 n\ge 1,\ \|h\|\le c\right\}
\]
are uniformly integrable under the corresponding laws $P_{n,h}$. Equivalently,
\[
\lim_{A\to\infty}\ \limsup_{n\to\infty}\
\sup_{\|h\|\le c}
\E_{n,h}\left[n\rho(\hat\mu_n,\mu_{n,h})^2
\mathbf 1\{n\rho(\hat\mu_n,\mu_{n,h})^2>A\}\right]=0 .
\]
\end{assumption}

\begin{theorem}[Local asymptotic minimaxity of the profiled location estimator]
\label{thm:lam-profile}
Assume that $X_1,X_2,\dots$ are i.i.d. from the local anisotropic HWN experiments defined above and that Assumption~\ref{ass:interior-info} holds. Let  $Z\sim N(0,\info_{\alpha\alpha\cdot\beta}^{-1})$. Then, for every fixed $A<\infty$ and $c<\infty$,
\begin{equation}
\label{eq:lam-attain-trunc}
\lim_{n\to\infty}\
\sup_{\|h\|\le c}
\left|
\E_{n,h}\Bigl[A\wedge n\rho\bigl(\hat\mu_n,\mu_{n,h}\bigr)^2\Bigr]
-
\E[A\wedge\|Z\|^2]
\right|=0.
\end{equation}
Consequently, the profile MLE attains the local asymptotic minimax bound for every truncated squared geodesic loss. Additionally, if Assumption~\ref{ass:ui} holds, then it also attains the untruncated squared-loss bound:
\begin{equation}
\label{eq:lam-attain}
\lim_{c\to\infty}\ \limsup_{n\to\infty}\
\sup_{\|h\|\le c}
\E_{n,h}\Bigl[n\rho\bigl(\hat\mu_n,\mu_{n,h}\bigr)^2\Bigr]
\le
\tr\bigl(\info_{\alpha\alpha\cdot\beta}^{-1}\bigr).
\end{equation}
Together with \eqref{eq:minimax-lower}, this yields local asymptotic minimaxity under ordinary squared geodesic loss whenever the stated uniform-integrability condition is satisfied.
\end{theorem}

The minimax statements proved above are local asymptotic statements, not exact finite-sample minimaxity on the whole manifold. The latter is much more delicate on noncompact curved parameter spaces and does not follow from symmetry alone. For the profile-MLE route, the LAN result of Theorem~\ref{thm:lan} and the \HajekLecam theorem provide the correct and standard decision-theoretic conclusion. The untruncated squared-loss attainment is deliberately separated from the truncated-loss result because it requires uniform integrability of local risks.

Thus the profile likelihood estimator has the expected optimality property for regular finite-dimensional likelihood problems. It is efficient in the LAN limit experiment and locally asymptotically minimax for the location parameter, with the covariance matrix treated as nuisance. The qualification concerning unbounded squared loss is purely technical and reflects the standard distinction between weak convergence and convergence of second moments.

\section{Computation}
\label{sec:computation}

The preceding sections established the statistical properties of the shell-constrained profile likelihood estimator. We now describe how the estimator is computed in practice. The main point is that the covariance parameter can be profiled out exactly for each candidate location, leaving a finite-dimensional optimization problem only in the location variable. The shell constraint, introduced for finite-sample existence and stability, is handled by a simple spectral clipping operation.

\subsection{Closed-form covariance profiling}

For a fixed location $\mu$, recall the empirical tangent scatter matrix
\[
S_n(\mu)
:=
\frac1n\sum_{i=1}^n
\T_\mu(X_i)\T_\mu(X_i)^\top .
\]
The covariance part of the negative profile log-likelihood is
\[
\Sigma\mapsto
\log\det\Sigma+\tr\{S_n(\mu)\Sigma^{-1}\},
\qquad
\Sigma\in\shell .
\]
Thus the nuisance optimization is a constrained covariance-fitting problem over the spectral shell. The next proposition gives its exact solution.

\begin{proposition}[Spectral form of the shell-constrained covariance profile]
\label{prop:spectral-clipping}
Let $S\in\Sym(d)$ be positive semidefinite and write its spectral decomposition as
\[
S=U\,\diag(s_1,\dots,s_d)\,U^\top,
\qquad s_j\ge 0.
\]
Then the unique minimizer of
\[
\Sigma\mapsto f_S(\Sigma)
=
\log\det\Sigma+\tr(S\Sigma^{-1})
\qquad\text{over}\qquad
\shell=\{\lambda_-\Id\preceq \Sigma\preceq \lambda_+\Id\}
\]
is
\begin{equation}
\label{eq:spectral-clipping}
\widetilde\Sigma(S)
=
U\,\diag(\widetilde s_1,\dots,\widetilde s_d)\,U^\top,
\qquad
\widetilde s_j
=
\min\{\lambda_+,\max\{\lambda_-,s_j\}\}.
\end{equation}
Equivalently, $\widetilde\Sigma(S)=g(S)$ by functional calculus, where
\[
g(s)=\min\{\lambda_+,\max\{\lambda_-,s\}\}.
\]
\end{proposition}

Following the above, the covariance profile obtained by clipping the eigenvalues of the empirical tangent covariance matrix. Thus each outer location iterate requires only the evaluation of tangent coordinates and one eigendecomposition of a $d\times d$ symmetric matrix.

\subsection{The profiled objective}

Combining the likelihood with Proposition~\ref{prop:spectral-clipping}, define
\[
\widetilde\Sigma_n(\mu)
:=
\widetilde\Sigma\{S_n(\mu)\}.
\]
Up to additive constants independent of $\mu$, the shell-constrained profiled negative log-likelihood is
\begin{equation}
\label{eq:shell-profile-objective}
Q_n^{\shell}(\mu)
=
\frac n2\log\det\widetilde\Sigma_n(\mu)
+
\frac n2\tr\{S_n(\mu)\widetilde\Sigma_n(\mu)^{-1}\}
+
(d-1)\sum_{i=1}^n
\psiwrap\{\T_\mu(X_i)\}.
\end{equation}
The shell-constrained profile estimator is any minimizer of
\[
\mu\mapsto Q_n^{\shell}(\mu),
\]
and the covariance estimate is then
\[
\hat\Sigma_n=\widetilde\Sigma_n(\hat\mu_n).
\]

When $S_n(\mu)$ is positive definite and lies in the interior of the shell, the clipping operation is inactive and $\widetilde\Sigma_n(\mu)=S_n(\mu)$. In that case,
\[
\tr\{S_n(\mu)\widetilde\Sigma_n(\mu)^{-1}\}=d,
\]
so the location-dependent part of the criterion reduces to
\begin{equation}
\label{eq:practical-objective}
Q_n(\mu)
=
\frac n2\log\det S_n(\mu)
+
(d-1)\sum_{i=1}^n
\psiwrap\{\T_\mu(X_i)\}.
\end{equation}
This is the practical profile objective used in unconstrained computations. By Proposition~\ref{prop:inactive-shell}, the shell-constrained and practical profile estimators coincide with probability tending to one under the correctly specified model whenever $\Sigma_0$ lies in the interior of the shell.

\subsection{Global coordinates for location optimization}

For computation,  it is convenient to parameterize the location by a global normal coordinate at the origin:
\begin{equation}
\label{eq:u-chart}
\nu=\Log_o(\mu)\in T_o\Hh^d\cong \R^d,
\qquad
\mu(\nu)=\Exp_o(\nu).
\end{equation}
Since $\Hh^d$ is Hadamard, the map $
\nu\mapsto\mu(\nu)
$
is a global diffeomorphism. Therefore the outer optimization may be carried out over the unconstrained Euclidean variable $\nu\in\R^d$.

In these coordinates, the shell-constrained objective becomes
\[
\nu\mapsto Q_n^{\shell}\{\mu(\nu)\},
\]
where the covariance profile is computed by applying Proposition~\ref{prop:spectral-clipping} to
\[
S_n\{\mu(\nu)\}
=
\frac1n\sum_{i=1}^n
\T_{\mu(\nu)}(X_i)\T_{\mu(\nu)}(X_i)^\top .
\]

\subsection{Optimization and numerical safeguards}

The Lorentz-model formulas for $\Exp_\mu$, $\Log_\mu$, and parallel transport are explicit \citep{nagano_2019_WrappedNormalDistribution}. Hence each evaluation of the profiled criterion consists of the following steps:
\begin{enumerate}[label=(\roman*)]
\item map the current Euclidean coordinate $\nu$ to $\mu(\nu)=\Exp_o(\nu)$;
\item compute tangent coordinates $\T_{\mu(\nu)}(X_i)$ for all observations;
\item form the empirical scatter matrix $S_n\{\mu(\nu)\}$;
\item compute the profiled covariance $\widetilde\Sigma_n\{\mu(\nu)\}$ by spectral clipping;
\item evaluate the profiled objective \eqref{eq:shell-profile-objective}.
\end{enumerate}

When the practical objective \eqref{eq:practical-objective} is used, its derivative formally satisfies
\[
D_\nu Q_n\{\mu(\nu)\}
=
\frac n2\,
\tr\!\left[
S_n\{\mu(\nu)\}^{-1}
D_\nu S_n\{\mu(\nu)\}
\right]
+
(d-1)\sum_{i=1}^n
D_\nu \psiwrap\{\T_{\mu(\nu)}(X_i)\},
\]
where
\[
D_\nu S_n\{\mu(\nu)\}
=
\frac1n\sum_{i=1}^n
D_\nu
\left[
\T_{\mu(\nu)}(X_i)\T_{\mu(\nu)}(X_i)^\top
\right].
\]
Hence, one may opt for standard algorithms including  quasi-Newton, trust-region, or Riemannian-gradient methods \citep{absil_2008_OptimizationAlgorithmsMatrix}. In our implementation, the optimization is performed in the global Euclidean coordinate $\nu$, with multiple initializations used in the more anisotropic scenarios.

Several numerical safeguards are important. First, the function
\[
r\mapsto \log\left(\frac{\sinh r}{r}\right)
\]
should not be evaluated naively at very small or very large $r$. For small $r$, we use the expansion
\[
\log\left(\frac{\sinh r}{r}\right)
=
\frac{r^2}{6}-\frac{r^4}{180}+O(r^6),
\qquad r\downarrow 0.
\]
For large $r$, we use the stable identity
\[
\log\left(\frac{\sinh r}{r}\right)
=
r-\log(2r)+\log(1-e^{-2r}).
\]
Second, log-determinants and quadratic forms are computed using eigendecompositions or Cholesky factors rather than explicit matrix inverses. Third, the shell profile is evaluated through eigenvalue clipping, which prevents singular covariance estimates in small samples or nearly degenerate configurations.

A natural initialization is the two-step intrinsic estimator to first compute a sample Fr\'echet mean $\hat\mu_F$ and then set
\[
\hat\Sigma_F
=
\frac1n\sum_{i=1}^n
\T_{\hat\mu_F}(X_i)\T_{\hat\mu_F}(X_i)^\top .
\]
The profile likelihood optimization is then initialized at $\hat\mu_F$, with optional random perturbations in the global coordinate $\nu$ for multi-start checks.

\section{Numerical study}
\label{sec:numerical}

We include one numerical study designed to assess the finite-sample calibration of the asymptotic theory. The purpose is not to compare a large collection of competing methods, but to examine whether the profile MLE behaves in accordance with the likelihood theory developed above. In particular, we compare the empirical scaled location risk with the efficient asymptotic target, examine Wald coverage for both the location coordinate and the full parameter, and record basic computational diagnostics for the profile likelihood optimization.

\subsection{Design}

Data are generated from the correctly specified anisotropic HWN model. For each replication,
\[
Z_i\sim N_d(0,\Sigma_0),
\qquad
X_i=\Exp_{\mu_0}\bigl(PT_{o\to\mu_0}Z_i\bigr),
\qquad i=1,\dots,n .
\]
The covariance matrix is a rotated anisotropic covariance with condition number \(10\). The location satisfies
\[
\rho(o,\mu_0)=3,
\]
so the experiment is not centered at the origin. We consider
\[
d\in\{2,5,10\},
\qquad
n\in\{100,250,500,1000\}.
\]
For each configuration we use \(500\) Monte Carlo replications. The Fisher information matrix and the efficient Schur-complement information are approximated by an independent Monte Carlo calculation with \(20{,}000\) draws. The covariance shell is fixed at
\[
0.03 I_d\preceq \Sigma\preceq 20 I_d .
\]

For each replication we compute the shell-constrained profile MLE. The primary location diagnostic is the scaled intrinsic risk
\[
n\,\rho(\hat\mu_n,\mu_0)^2 .
\]
Its Monte Carlo average is compared with the efficient asymptotic target
\[
\tr\bigl(\info_{\alpha\alpha\cdot\beta}^{-1}\bigr).
\]
We also report the empirical coverage of the nominal \(95\%\) Wald ellipsoid for the location coordinate, based on \(\info_{\alpha\alpha\cdot\beta}^{-1}\), and the empirical coverage of the nominal \(95\%\) Wald ellipsoid for the full local parameter \((\alpha,\beta)\), based on \(\info(\vartheta_0)^{-1}\). Finally, to assess covariance calibration of the location limit distribution, we report
\[
\frac{
\left\|
\widehat{\Cov}\{\sqrt n\,\Log_{\mu_0}(\hat\mu_n)\}
-
\info_{\alpha\alpha\cdot\beta}^{-1}
\right\|_F
}{
\left\|
\info_{\alpha\alpha\cdot\beta}^{-1}
\right\|_F
}.
\]
The shell activity column records the fraction of replications for which the profiled covariance estimate lies on the boundary of the covariance shell.

\subsection{Results}

Table~\ref{tab:numerical-calibration} reports the results. The scaled location risk is close to the efficient target across all dimensions and sample sizes. The ratio between empirical scaled risk and the target ranges from \(0.967\) to \(1.104\), which is consistent with the prediction
\[
n\,E\{\rho(\hat\mu_n,\mu_0)^2\}
\approx
\tr\bigl(\info_{\alpha\alpha\cdot\beta}^{-1}\bigr).
\]
The location Wald coverage is close to the nominal level, ranging from \(0.922\) to \(0.962\). The full-parameter coverage is more sensitive in small samples, especially in dimension \(10\), where the nuisance covariance coordinate is high-dimensional; nevertheless it improves as \(n\) increases and is close to nominal in the larger-sample cases.

The relative covariance calibration error is small in dimension \(2\) once \(n\ge 250\), and remains moderate in dimensions \(5\) and \(10\). Since the information matrix is itself estimated by Monte Carlo, this quantity should be interpreted as a combined finite-sample and numerical calibration diagnostic rather than as an exact discrepancy. The covariance shell is inactive in all reported configurations. This supports the theoretical conclusion in Proposition~\ref{prop:inactive-shell}: the shell is useful for finite-sample well-posedness, but does not bind when the true covariance lies well inside it.

\begin{table}[t]
\centering
\caption{Finite-sample calibration of the profile MLE under the anisotropic HWN model. The covariance has condition number \(10\), the location satisfies \(\rho(o,\mu_0)=3\), and the covariance orientation is rotated. The column \(n\hat R_\mu\) is the Monte Carlo average of \(n\rho(\hat\mu_n,\mu_0)^2\). The target is \(\tr(\info_{\alpha\alpha\cdot\beta}^{-1})\). Location and full coverage refer to nominal \(95\%\) Wald ellipsoids.}
\label{tab:numerical-calibration}
\small
\begin{tabular}{rrrrrrrrr}
\toprule
\(d\) & \(n\) & \(n\hat R_\mu\) & Target & Ratio
& Loc. cov. & Full cov. & Rel. cov. err. & Shell active \\
\midrule
2  & 100  & 1.446 & 1.310 & 1.104 & 0.940 & 0.928 & 0.112 & 0.000 \\
2  & 250  & 1.296 & 1.310 & 0.990 & 0.946 & 0.948 & 0.012 & 0.000 \\
2  & 500  & 1.312 & 1.310 & 1.002 & 0.948 & 0.954 & 0.006 & 0.000 \\
2  & 1000 & 1.335 & 1.310 & 1.019 & 0.950 & 0.944 & 0.027 & 0.000 \\
\midrule
5  & 100  & 1.332 & 1.222 & 1.089 & 0.922 & 0.884 & 0.126 & 0.000 \\
5  & 250  & 1.206 & 1.222 & 0.986 & 0.950 & 0.942 & 0.061 & 0.000 \\
5  & 500  & 1.204 & 1.222 & 0.985 & 0.954 & 0.940 & 0.063 & 0.000 \\
5  & 1000 & 1.288 & 1.222 & 1.054 & 0.940 & 0.940 & 0.102 & 0.000 \\
\midrule
10 & 100  & 1.204 & 1.156 & 1.042 & 0.952 & 0.854 & 0.117 & 0.000 \\
10 & 250  & 1.166 & 1.156 & 1.009 & 0.942 & 0.908 & 0.124 & 0.000 \\
10 & 500  & 1.118 & 1.156 & 0.967 & 0.958 & 0.948 & 0.117 & 0.000 \\
10 & 1000 & 1.129 & 1.156 & 0.977 & 0.962 & 0.954 & 0.104 & 0.000 \\
\bottomrule
\end{tabular}
\end{table}

Figure~\ref{fig:numerical-risk-target} visualizes the same calibration by plotting the scaled location risk against the efficient target. The agreement is already close at moderate sample sizes and remains stable across the larger sample sizes. Small nonmonotonic fluctuations are expected because the experiment uses finite Monte Carlo replication and a Monte Carlo approximation to the Fisher information.

\begin{figure}[t]
\centering
\includegraphics[width=0.72\textwidth]{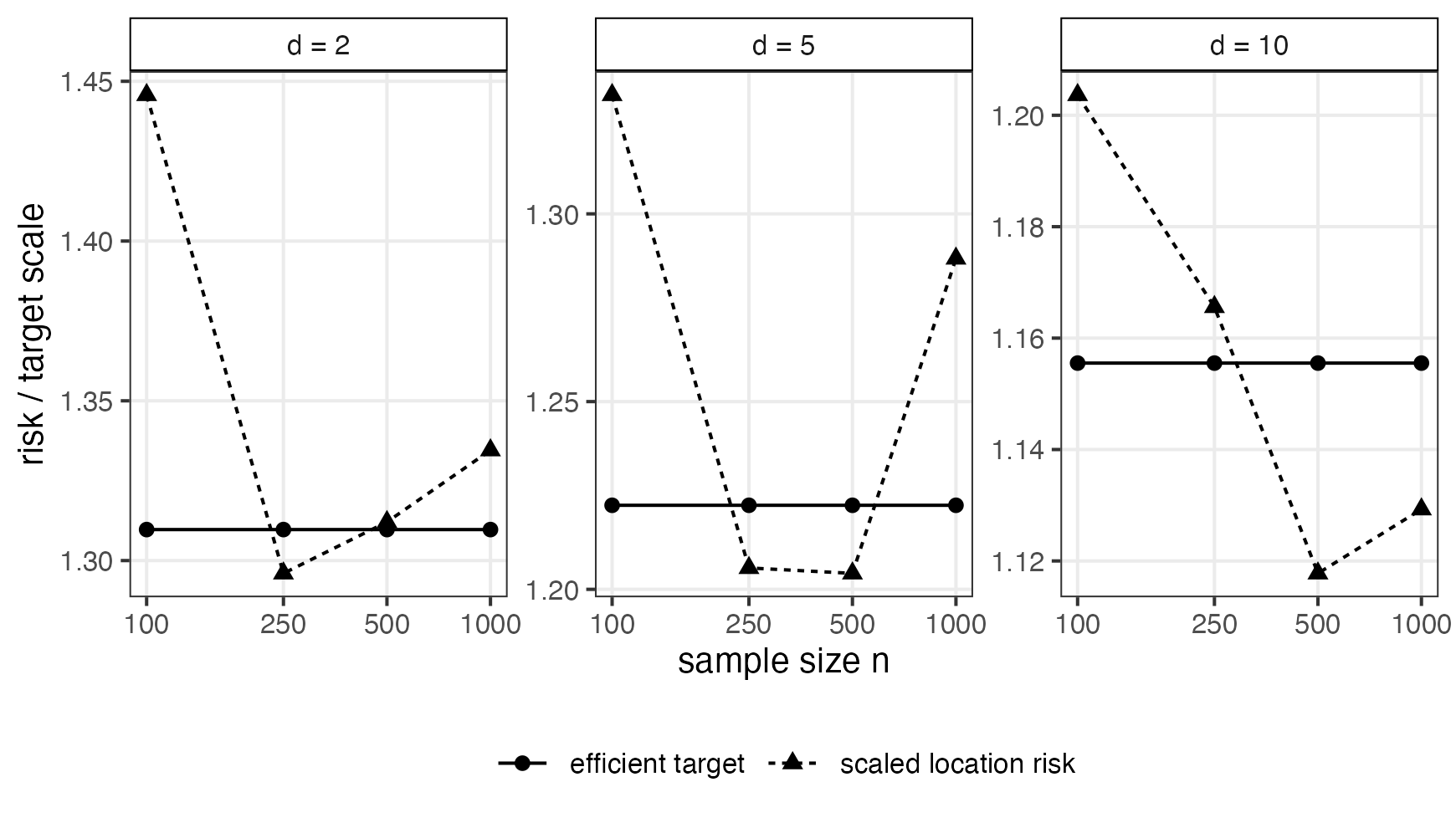}
\caption{Scaled location risk \(n\,E\{\rho(\hat\mu_n,\mu_0)^2\}\) compared with the efficient asymptotic target \(\tr(\info_{\alpha\alpha\cdot\beta}^{-1})\).}
\label{fig:numerical-risk-target}
\end{figure}

Figure~\ref{fig:numerical-coverage} reports the empirical coverage of nominal
\(95\%\) Wald regions. The location coverage remains close to nominal across the
reported configurations. The full-parameter coverage is more sensitive in small
samples, especially in dimension \(10\), where the nuisance covariance coordinate
is high-dimensional, but it improves as the sample size increases.

\begin{figure}[t]
\centering
\includegraphics[width=0.72\textwidth]{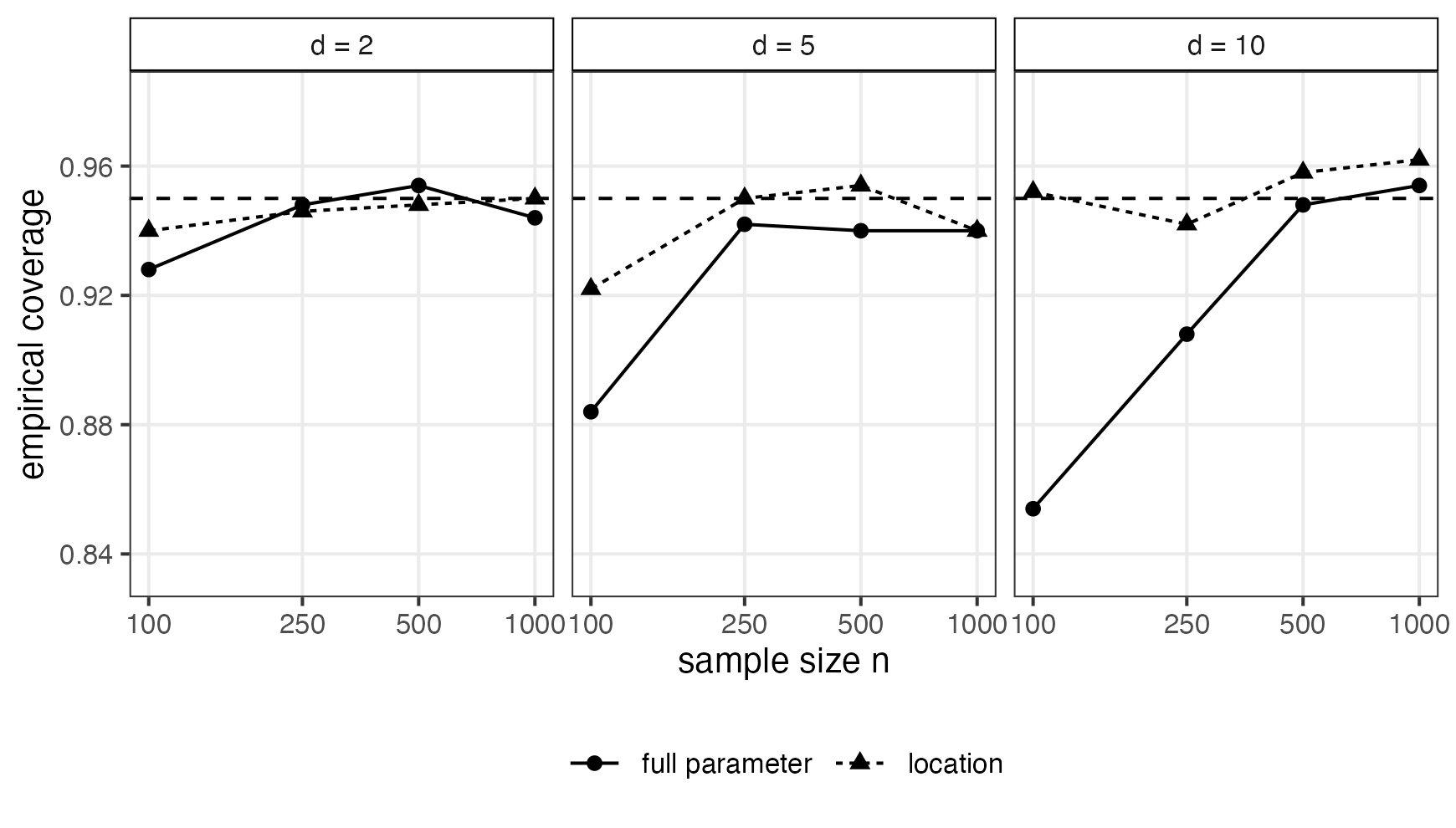}
\caption{Empirical coverage of nominal \(95\%\) Wald regions. The location
coverage uses the Schur-complement covariance
\(\info_{\alpha\alpha\cdot\beta}^{-1}\), while the full-parameter coverage uses
the inverse Fisher information \(\info(\vartheta_0)^{-1}\).}
\label{fig:numerical-coverage}
\end{figure}

The computation was also stable. Median runtimes ranged from approximately \(0.10\) seconds in dimension \(2\), \(n=100\), to approximately \(1.46\) seconds in dimension \(10\), \(n=1000\). Median iteration counts were roughly \(32\), \(61\), and \(108\) in dimensions \(2\), \(5\), and \(10\), respectively. These values are consistent with the fact that the outer optimization is only over the \(d\)-dimensional location coordinate, while the covariance nuisance parameter is profiled out explicitly by spectral clipping.

Overall, the numerical results support the second-order theory. The empirical scaled risks track the efficient Schur-complement target, the location Wald coverage is close to nominal, the full-parameter coverage improves with sample size, and the covariance shell is inactive in all reported configurations. These findings are consistent with the role of the shell as a finite-sample regularization device rather than an asymptotic constraint.

\section{Conclusion}
\label{sec:conclusion}

This paper develops a likelihood-based inferential theory for the anisotropic hyperbolic wrapped normal distribution on standard hyperbolic space. The model combines a manifold-valued location parameter with a full tangent-space covariance matrix, and therefore requires both geometric and multivariate likelihood arguments. The estimator studied here is the profile maximum likelihood estimator obtained by optimizing over the location after profiling out the covariance parameter.

A technical difficulty in the anisotropic model is that the unrestricted covariance profile may fail to attain a finite maximum when the empirical tangent covariance is singular. We address this by introducing a compact covariance shell whose eigenvalue bounds keep the likelihood well posed. This device guarantees finite-sample existence of the constrained estimator, while remaining asymptotically inactive when the true covariance lies in the interior of the shell. Thus the constrained formulation provides a rigorous foundation for the practical profile likelihood computation.

The main theoretical results establish identifiability, existence of the shell-constrained joint and profile maximum likelihood estimators, strong consistency, joint asymptotic normality, and efficient profile asymptotic normality for the location parameter. The covariance parameter is treated as nuisance for location inference, and the efficient asymptotic covariance is given by the inverse Schur complement of the nuisance information block. We also prove local asymptotic normality of the anisotropic HWN experiment and derive the corresponding Hajek-Le Cam local asymptotic minimax lower bound under squared geodesic loss. The profile estimator attains the bound for truncated squared geodesic loss, and for ordinary squared geodesic loss under a stated local uniform-integrability condition.

The computational form of the estimator is explicit. For each candidate location, the covariance profile is obtained by spectral clipping of the empirical tangent covariance matrix. The remaining optimization is over the location only and can be carried out in global normal coordinates at the origin. The numerical study supports the second-order theory: the empirical scaled location risk tracks the Schur-complement target, Wald coverage is close to nominal, and the covariance shell is inactive in the reported configurations.

Several extensions remain open. One natural direction is to move beyond simply connected hyperbolic space to more general negatively curved manifolds, where cut-locus and chart issues may arise. Another is to study robustness under misspecification, for example when the tangent distribution is heavy-tailed or when the data arise from mixtures of wrapped normal components. A third direction is to develop high-dimensional theory in which the covariance dimension grows with the sample size.




\bibliographystyle{dcu}
\bibliography{references}

\appendix
\section{Proofs}

\noindent \textbf{Proof of Proposition~\ref{prop:density}}. Let $Z\sim N_d(0,\Sigma)$ and define $F_\mu(v)=\Exp_\mu(PT_{o\to\mu}v)$. By construction $X=F_\mu(Z)$. Parallel transport is an isometry, hence has Jacobian determinant $1$, while the Jacobian determinant of $\Exp_\mu$ at a tangent vector of norm $r$ is $(\sinh r/r)^{d-1}$ in the Lorentz model \citet[Sec.~3.2]{nagano_2019_WrappedNormalDistribution}. Therefore
\[
\bigl|\det D F_\mu(v)\bigr|
=
\Bigl(\frac{\sinh \norm{v}}{\norm{v}}\Bigr)^{d-1}.
\]
Applying the change-of-variables formula with $v=\T_\mu(x)$ yields \eqref{eq:density}. Continuity at $x=\mu$ follows from $\sinh r/r\to 1$ as $r\to 0$.

\medskip
\noindent \textbf{Proof of Proposition~\ref{prop:conditionalSigma}}. Existence on $\shell$ is immediate from continuity and compactness. Assume now that $S\in\SPD(d)$. For any $\Sigma\in\SPD(d)$, let $A=S^{1/2}\Sigma^{-1}S^{1/2}$. Then
\begin{align*}
f_S(\Sigma)-f_S(S)
&= \log\det\Sigma-\log\det S+\tr(S\Sigma^{-1})-d\\
&= -\log\det A+\tr(A)-d\\
&= \sum_{j=1}^d \bigl(\lambda_j(A)-\log\lambda_j(A)-1\bigr)\ge 0,
\end{align*}
by the scalar inequality $u-\log u-1\ge 0$ for $u>0$. Equality holds if and only if every eigenvalue of $A$ equals $1$, that is, if and only if $A=\Id$ or equivalently $\Sigma=S$. Hence $S$ is the unique minimizer on $\SPD(d)$. If $S\in \mathrm{int}(\shell)$, the same uniqueness shows that the shell-constrained minimizer is $S$.

\medskip 
\noindent \textbf{Proof of Lemma~\ref{lem:moments}}. Write $X=\Exp_{\mu}(PT_{o\to \mu}Z)$ with $Z\sim N_d(0,\Sigma)$. Because geodesic distance equals tangent norm under the exponential map,  $\rho(\mu,X)=\norm{Z}$. Hence, by the triangle inequality,
\[
\rho(o,X)\le \rho(o,\mu)+\rho(\mu,X)=\rho(o,\mu)+\norm{Z}.
\]
For a fixed parameter this proves part (a), since all moments of the Euclidean Gaussian norm are finite. For part (b), let $\lambda_{\max}$ be the largest eigenvalue of $\Sigma$ and choose any $t\in\bigl(0,(2\lambda_{\max})^{-1}\bigr)$. Then for any $a>0$ and $r\ge 0$,
\[
ar\le tr^2+\frac{a^2}{4t},
\]
so
\[
\E e^{a\norm{Z}}\le e^{a^2/(4t)}\E e^{t\norm{Z}^2}<\infty.
\]
The claim follows by multiplying by $e^{a\rho(o,\mu)}$. The same argument is uniform on $K_\mu\times K_\Sigma$: replace $\rho(o,\mu)$ by $\sup_{\mu\in K_\mu}\rho(o,\mu)$ and choose $t$ smaller than one half of the reciprocal of the uniform upper eigenvalue bound on $K_\Sigma$.

\medskip
\noindent \textbf{Proof of Proposition~\ref{prop:envelope}}. 
We divide the proof into four steps.

\smallskip
\noindent\emph{Step 1: smoothness of the basic coordinate map.}
Consider the global map
\[
F:\Hh^d\times\R^d\to \Hh^d\times \Hh^d,
\qquad
F(\mu,v)=\bigl(\mu,\Exp_\mu(PT_{o\to \mu}v)\bigr).
\]
Because $PT_{o\to\mu}$ and $\Exp_\mu$ are smooth in $(\mu,v)$ and $\Exp_\mu$ is a global diffeomorphism on each tangent space, $F$ is a smooth global diffeomorphism. Its inverse has the form
\[
F^{-1}(\mu,x)=\bigl(\mu,\T_\mu(x)\bigr).
\]
Hence $(\mu,x)\mapsto \T_\mu(x)$ is smooth on all of $\Hh^d\times\Hh^d$.

\smallskip
\noindent\emph{Step 2: polynomial growth of derivatives of $\T_\mu(x)$ in $\mu$.}
Fix a compact $K_\mu\subset\Hh^d$. We claim that for $j=0,1,2,3$,
\begin{equation}
\label{eq:T-growth}
\sup_{\mu\in K_\mu}\bigl\|\partial_\alpha^j \T_\mu(x)\bigr\|
\le C_j'(1+\rho(o,x)^{q_j'}),
\qquad x\in\Hh^d.
\end{equation}
Indeed, because $(\mu,x)\mapsto \T_\mu(x)$ is smooth, its derivatives are bounded on the compact set
\[
\{(\mu,x):\mu\in K_\mu,\ \rho(\mu,x)\le 1\},
\]
which is compact because $K_\mu$ is compact and $\Hh^d$ is proper.

It remains to control the region $\rho(\mu,x)\ge 1$. Write $r=\rho(\mu,x)$ and recall the explicit formula
\[
\Log_\mu(x)=q(r)\bigl(x-\cosh(r)\mu\bigr),
\qquad
q(r):=\frac{r}{\sinh r}.
\]
On $[1,\infty)$ the functions $q,q',q'',q'''$ are bounded by constants times $(1+r)e^{-r}$. Moreover,
\[
\norm{x-\cosh(r)\mu}_{T_\mu\Hh^d}=\sinh r.
\]
Differentiating the display above up to order three with respect to $\mu$ therefore produces finite sums of terms built out of:
(i) derivatives of $q(r)$, each bounded by $(1+r)e^{-r}$ up to a constant;
(ii) derivatives of $\cosh r$ or $\sinh r$;
(iii) derivatives of the distance function $r=\rho(\mu,x)$ with respect to $\mu$;
(iv) factors $x-\cosh(r)\mu$ or $\mu$.
For $r\ge 1$, the first derivative of $r$ has norm $1$, the Hessian of $r$ is
\[
\nabla^2 r = \coth(r)\bigl(g-dr\otimes dr\bigr),
\]
and therefore has operator norm bounded by $\coth(1)$, while the third derivative is uniformly bounded on $\{r\ge 1\}$ because differentiating the Hessian formula on the constant-curvature manifold $\Hh^d$ only introduces smooth functions of $r$ with bounded derivatives on $[1,\infty)$. Hence every term arising from differentiation is bounded by a constant times a polynomial in $r$. Since parallel transport $PT_{o\to\mu}^{-1}$ and its derivatives are bounded on compact $K_\mu$, \eqref{eq:T-growth} follows.

Finally, because
\[
\rho(\mu,x)\le \rho(\mu,o)+\rho(o,x)\le C_{K_\mu}+\rho(o,x),
\]
the right-hand side of \eqref{eq:T-growth} can be expressed as a polynomial in $\rho(o,x)$.

\smallskip
\noindent\emph{Step 3: growth of the Jacobian term.}
The map $\psiwrap:\R^d\to\R$ defined in \eqref{eq:phi-def} is smooth and radial. For large $\norm{v}$,
\[
\psiwrap(v)=\norm{v}-\log 2-\log \norm{v}+O(e^{-2\norm{v}}),
\]
hence its gradient is bounded and its higher derivatives are $O(\norm{v}^{-1})$ and $O(\norm{v}^{-2})$, respectively. Near $0$, smoothness gives boundedness of all derivatives on bounded sets. Therefore, for $j=0,1,2,3$,
\[
\bigl\|\partial^j \psiwrap(v)\bigr\|
\le C_j''(1+\norm{v}).
\]
Combining this bound with \eqref{eq:T-growth} and the multivariate chain rule yields a polynomial envelope for derivatives of $\mu\mapsto \psiwrap(\T_\mu(x))$ up to order three.

\smallskip
\noindent\emph{Step 4: the covariance chart.}
On a compact set $B\subset\R^m$, the matrix-valued maps
\[
\beta\mapsto \Sigma(\beta),\qquad
\beta\mapsto \Sigma(\beta)^{-1},\qquad
\beta\mapsto \log\det \Sigma(\beta)
\]
and their derivatives up to order three are bounded, because the matrix exponential and inversion are smooth on finite-dimensional spaces and $B$ is compact \citep{bhatia_2009_PositiveDefiniteMatrices}. The quadratic term
\[
(\mu,\beta,x)\mapsto \T_\mu(x)^\top \Sigma(\beta)^{-1}\T_\mu(x)
\]
is therefore a smooth composition whose derivatives up to order three are bounded by a polynomial in $\norm{\T_\mu(x)}$ and in the derivatives of $\T_\mu(x)$, hence by a polynomial in $\rho(o,x)$ thanks to \eqref{eq:T-growth}. The same holds for the full log-density \eqref{eq:single-loglik}. This proves \eqref{eq:envelope}.

The stated moment consequences follow from Lemma~\ref{lem:moments}, since every polynomial in $\rho(o,X)$ has finite moments of all orders under $P_{\vartheta_0}$.

\medskip
\noindent \textbf{Proof of Lemma~\ref{lem:scoremean}}. By Proposition~\ref{prop:envelope},  derivatives of $p_\vartheta(x)$ with respect to $\vartheta$ are dominated on compact neighborhoods by integrable envelopes. Therefore differentiation can be passed under the integral sign:
\[
0=\partial_\vartheta \int p_\vartheta(x)\,d\mathrm{vol}(x)=\int \partial_\vartheta p_\vartheta(x)\,d\mathrm{vol}(x)
=\int \score_\vartheta(x)p_\vartheta(x)\,d\mathrm{vol}(x).
\]
This yields $\E_\vartheta[\score_\vartheta(X)]=0$. Differentiating once more and using the identity
\[
\partial_\vartheta^2 p_\vartheta
=
\bigl(\partial_\vartheta^2 m_\vartheta+\score_\vartheta\score_\vartheta^\top\bigr)p_\vartheta
\]
gives the information identity.

\medskip
\noindent \textbf{Proof of Lemma~\ref{lem:ulln}}. By Proposition~\ref{prop:envelope}, the map $\vartheta\mapsto m_\vartheta(x)$ is continuously differentiable on $K$ and
\[
L_K(x):=\sup_{\vartheta\in K}\norm{\partial_\vartheta m_\vartheta(x)}
\]
has finite expectation under $P_{\vartheta_0}$. Hence the mean-value theorem gives
\[
|m_{\vartheta_1}(x)-m_{\vartheta_2}(x)|\le L_K(x)\norm{\vartheta_1-\vartheta_2},
\qquad \vartheta_1,\vartheta_2\in K.
\]
Choose an $\varepsilon$-net $\{\vartheta^{(1)},\dots,\vartheta^{(N)}\}$ of $K$. For any $\vartheta\in K$ select $k(\vartheta)$ with $\norm{\vartheta-\vartheta^{(k(\vartheta))}}\le\varepsilon$. Then
\begin{align*}
|M_n(\vartheta)-M(\vartheta)|
&\le |M_n(\vartheta^{(k)})-M(\vartheta^{(k)})|
+ \varepsilon\Bigl(\frac1n\sum_{i=1}^n L_K(X_i)+\E L_K(X)\Bigr).
\end{align*}
Taking suprema over $K$ and then limsup, the strong law of large numbers gives
\[
\limsup_{n\to\infty}\sup_{\vartheta\in K}|M_n(\vartheta)-M(\vartheta)|
\le 2\varepsilon \E L_K(X)
\qquad \text{a.s.}
\]
Because $\varepsilon>0$ is arbitrary, the conclusion follows. For first derivatives, the same argument uses the second-derivative envelope and for second derivatives, it uses the third-derivative envelope. We do neither need nor assert a compact-uniform law for third derivatives.

\medskip 
\noindent \textbf{Proof of Theorem~\ref{thm:existence}}. We first prove coercivity in the location parameter. For any $(\mu,\Sigma)\in\Theta$, using $\Sigma^{-1}\succeq \lambda_+^{-1}\Id$ and $\det \Sigma\ge \lambda_-^d$, \eqref{eq:sample-loglik2} yields
\begin{equation}
\label{eq:coercive-upper}
\ell_n(\mu,\Sigma)
\le
-\frac{nd}{2}\log\lambda_-
-\frac{1}{2\lambda_+}\sum_{i=1}^n r_\mu(X_i)^2
-(d-1)\sum_{i=1}^n \psiwrap(\T_\mu(X_i)).
\end{equation}
In particular, since $\psiwrap\ge 0$,
\[
\ell_n(\mu,\Sigma)
\le C_n - \frac{1}{2\lambda_+}\sum_{i=1}^n r_\mu(X_i)^2
\]
for a finite sample-dependent constant $C_n$.

Let $D(\mu):=\rho(o,\mu)$ and set $\bar d_n:=n^{-1}\sum_{i=1}^n \rho(o,X_i)$. By the triangle inequality,
\[
r_\mu(X_i)=\rho(\mu,X_i)\ge D(\mu)-\rho(o,X_i),
\]
hence by Jensen,
\[
\frac1n\sum_{i=1}^n r_\mu(X_i)^2
\ge
\Bigl(\frac1n\sum_{i=1}^n r_\mu(X_i)\Bigr)^2
\ge
\bigl(D(\mu)-\bar d_n\bigr)_+^2,
\qquad a_+:=\max\{a,0\}.
\]
Substituting into \eqref{eq:coercive-upper} gives the uniform bound
\begin{equation}
\label{eq:uniform-coercive}
\sup_{\Sigma\in\shell}\ell_n(\mu,\Sigma)
\le
C_n - \frac{n}{2\lambda_+}\bigl(D(\mu)-\bar d_n\bigr)_+^2.
\end{equation}
Therefore
\[
\sup_{\Sigma\in\shell}\ell_n(\mu,\Sigma)\to -\infty
\qquad\text{as }D(\mu)\to\infty.
\]
So every maximizer of the likelihood must lie in a closed geodesic ball $
B_R(o)=\{\mu\in\Hh^d:\rho(o,\mu)\le R\}
$ for $R$ large enough.

Now $B_R(o)$ is compact because $\Hh^d$ is proper \citep{petersen_2006_RiemannianGeometry}  and $\shell$ is compact by definition. The log-likelihood is continuous by Proposition~\ref{prop:envelope}. Hence the restriction of $\ell_n$ to $B_R(o)\times\shell$ attains its maximum. By construction, the same point maximizes $\ell_n$ over all of $\Theta$. Existence of the profile maximizer follows immediately.

\medskip 
\noindent \textbf{Proof of Theorem~\ref{thm:identifiability}}. Fix $(\mu,\Sigma)$. By \eqref{eq:density},
\[
p_{\mu,\Sigma}(x)
=
(2\pi)^{-d/2}|\Sigma|^{-1/2}
\exp\;\Bigl(-\tfrac12\T_\mu(x)^\top\Sigma^{-1}\T_\mu(x)\Bigr)
\Bigl(\frac{r_\mu(x)}{\sinh r_\mu(x)}\Bigr)^{d-1}.
\]
At $x=\mu$ we have $\T_\mu(\mu)=0$ and $r_\mu(\mu)=0$, so
\[
p_{\mu,\Sigma}(\mu)=(2\pi)^{-d/2}|\Sigma|^{-1/2}.
\]
If $x\neq \mu$, then $r_\mu(x)>0$, hence
\[
\exp\;\Bigl(-\tfrac12\T_\mu(x)^\top\Sigma^{-1}\T_\mu(x)\Bigr)<1
\qquad\text{and}\qquad
\frac{r_\mu(x)}{\sinh r_\mu(x)}<1.
\]
Therefore $p_{\mu,\Sigma}(x)<p_{\mu,\Sigma}(\mu)$ for every $x\neq \mu$, so $\mu$ is the unique mode of the density.

If two parameter pairs $(\mu,\Sigma)$ and $(\mu',\Sigma')$ define the same distribution, their densities coincide volume-almost everywhere. Since both densities are continuous, equality almost everywhere implies equality everywhere. Otherwise, continuity would give a nonempty open set on which they differ and every nonempty open set in $\Hh^d$ has positive Riemannian volume. The two densities therefore have the same unique mode, and hence $\mu=\mu'$. Once the location is fixed, the random vector $\T_\mu(X)$ is exactly Gaussian $N(0,\Sigma)$ by construction. Equality of the laws therefore forces $\Sigma=\Sigma'$, which completes the proof.

\medskip
\noindent \textbf{Proof of Theorem~\ref{thm:consistency}}.  Write $\theta=(\mu,\Sigma)$ and define the population criterion
\[
M(\theta):=\E_{\theta_0}[m_\theta(X)].
\]
We divide the proof into four steps.

\smallskip
\noindent\emph{Step 1: Kullback--Leibler identification.}
By Proposition~\ref{prop:envelope}, $m_\theta(X)$ is integrable under $P_{\theta_0}$ for every $\theta$. Hence
\[
M(\theta)-M(\theta_0)
=
\E_{\theta_0}\log\frac{p_\theta(X)}{p_{\theta_0}(X)}
=
-\KL(P_{\theta_0}\,\|\,P_\theta)\le 0.
\]
By Theorem~\ref{thm:identifiability}, equality holds if and only if $\theta=\theta_0$. Thus $M$ has the unique maximizer $\theta_0$.

\smallskip
\noindent\emph{Step 2: localization on a deterministic compact set.}
By the strong law and Lemma~\ref{lem:moments}, $\bar d_n:=n^{-1}\sum_{i=1}^n \rho(o,X_i)\to \E\rho(o,X)<\infty$ almost surely. Fix a finite constant $C>\E\rho(o,X)+1$. On the almost sure event on which $\bar d_n\le C$ eventually, the coercive bound \eqref{eq:uniform-coercive} implies that eventually
\[
\sup_{\Sigma\in\shell}\ell_n(\mu,\Sigma)
\le
C_n-\frac{n}{2\lambda_+}(D(\mu)-C)_+^2.
\]
On the other hand, by the ordinary strong law,
\[
\frac1n\ell_n(\mu_0,\Sigma_0)\to M(\theta_0)
\qquad\text{a.s.}
\]
Therefore, almost surely, there exists $R<\infty$ such that for all sufficiently large $n$,
\[
\sup_{\rho(o,\mu)>R,\ \Sigma\in\shell}\frac1n\ell_n(\mu,\Sigma)
<
\frac1n\ell_n(\mu_0,\Sigma_0)-1.
\]
Because $\hat\theta_n$ is a maximizer, it follows that $\hat\theta_n$ eventually belongs almost surely to the compact set
\[
K_R:=\{(\mu,\Sigma):\rho(o,\mu)\le R,\ \Sigma\in\shell\}.
\]

\smallskip
\noindent\emph{Step 3: uniform convergence on $K_R$.}
Choose the Euclidean chart $(\alpha,\beta)$ and let $K_R^\flat$ denote the image of $K_R$ in $\R^p$. The set $K_R^\flat$ is compact. By Lemma~\ref{lem:ulln}, 
\[
\sup_{\vartheta\in K_R^\flat}|M_n(\vartheta)-M(\vartheta)|\to 0
\qquad\text{a.s.}
\]
where $M_n(\vartheta)=n^{-1}\ell_n(\vartheta)$.

\smallskip
\noindent\emph{Step 4: the argmax conclusion.}
By Step 2, $\hat\vartheta_n$ lies in $K_R^\flat$ eventually. By Step 3, $M_n$ converges uniformly to $M$ on that compact set. By Step 1, $M$ has a unique maximizer at $\vartheta_0$. The standard compact argmax theorem therefore yields
\[
\hat\vartheta_n\to \vartheta_0
\qquad\text{a.s.}
\]
which is the desired conclusion.

\medskip
\noindent \textbf{Proof of Proposition~\ref{prop:inactive-shell}}. Let $S(\mu):=\E_{\theta_0}\bigl[\T_\mu(X)\T_\mu(X)^\top\bigr]$. By Proposition~\ref{prop:envelope}, the matrix-valued function $\mu\mapsto \T_\mu(x)\T_\mu(x)^\top$ is continuous and dominated on compact sets by an integrable envelope. Hence $S(\mu)$ is continuous, and the argument of Lemma~\ref{lem:ulln} applied entrywise gives
\[
\sup_{\mu\in K_\mu}\norm{S_n(\mu)-S(\mu)}\to 0
\qquad\text{a.s.}
\]
for every compact $K_\mu\subset\Hh^d$.

By Theorem~\ref{thm:consistency}, $\hat\mu_n\to\mu_0$ almost surely, so eventually $\hat\mu_n$ lies in some compact $K_\mu$. Therefore
\[
\norm{S_n(\hat\mu_n)-S(\mu_0)}
\le
\sup_{\mu\in K_\mu}\norm{S_n(\mu)-S(\mu)}
+\norm{S(\hat\mu_n)-S(\mu_0)}
\to 0
\qquad\text{a.s.}
\]
But under the true model, $\T_{\mu_0}(X)=Z\sim N(0,\Sigma_0)$, hence
\[
S(\mu_0)=\E[ZZ^\top]=\Sigma_0.
\]
Thus $S_n(\hat\mu_n)\to \Sigma_0$ almost surely. Since $\Sigma_0\in \mathrm{int}(\shell)$, the event $S_n(\hat\mu_n)\in \mathrm{int}(\shell)$ occurs eventually almost surely. The last claim then follows from Proposition~\ref{prop:conditionalSigma}.

\medskip
\noindent \textbf{Proof of Theorem~\ref{thm:joint-asym-normal}}. Since $\Sigma_0\in \mathrm{int}(\shell)$ and $\hat\Sigma_n\to \Sigma_0$ almost surely by Theorem~\ref{thm:consistency}, the estimator lies in the interior of the shell with probability tending to one. On that event the first-order optimality equation holds:
\[
0=\partial_\vartheta M_n(\hat\vartheta_n)
=
\frac1n\sum_{i=1}^n \score_{\hat\vartheta_n}(X_i).
\]
Use the integral form of the multivariate Taylor expansion. On the event above,
\[
0
=
\frac1n\sum_{i=1}^n \score_{\vartheta_0}(X_i)
+
H_n(\hat\vartheta_n)(\hat\vartheta_n-\vartheta_0),
\]
where
\[
H_n(\hat\vartheta_n)
:=
\int_0^1
\frac1n\sum_{i=1}^n
\dot\score_{\vartheta_0+t(\hat\vartheta_n-\vartheta_0)}(X_i)\,dt .
\]
This integral expansion is used instead of a scalar mean-value form, which is not valid for vector-valued scores in general. Multiplying by $\sqrt n$ and rearranging gives
\begin{equation}
\label{eq:taylor-rootn}
\sqrt n(\hat\vartheta_n-\vartheta_0)
=
-
H_n(\hat\vartheta_n)^{-1}
\frac1{\sqrt n}\sum_{i=1}^n \score_{\vartheta_0}(X_i).
\end{equation}
Because $\hat\vartheta_n\to \vartheta_0$ in probability, the entire line segment
$\{\vartheta_0+t(\hat\vartheta_n-\vartheta_0):0\le t\le 1\}$ is contained with probability tending to one in any fixed compact neighborhood $K$ of $\vartheta_0$. By Lemma~\ref{lem:ulln},
\[
\sup_{\vartheta\in K}\Bigl\|
\frac1n\sum_{i=1}^n \dot\score_\vartheta(X_i)-\E_{\vartheta_0}\dot\score_\vartheta(X)
\Bigr\|\to 0
\qquad\text{a.s.}
\]
The map $\vartheta\mapsto \E_{\vartheta_0}\dot\score_\vartheta(X)$ is continuous by dominated convergence and Proposition~\ref{prop:envelope}. Hence
\[
H_n(\hat\vartheta_n)
\to
\E_{\vartheta_0}\dot\score_{\vartheta_0}(X)
=
-\info(\vartheta_0)
\qquad\text{in probability},
\]
where the last identity is Lemma~\ref{lem:scoremean}. Since $\info(\vartheta_0)$ is nonsingular by Assumption~\ref{ass:interior-info}, the inverse in \eqref{eq:taylor-rootn} is well-defined with probability tending to one and converges to $-\info(\vartheta_0)^{-1}$ in probability.

Finally, the multivariate central limit theorem applies to the score because Proposition~\ref{prop:envelope} supplies moments of order $2+\delta$:
\[
\frac1{\sqrt n}\sum_{i=1}^n \score_{\vartheta_0}(X_i)\Rightarrow N\bigl(0,\info(\vartheta_0)\bigr).
\]
Substituting into \eqref{eq:taylor-rootn} and using Slutsky's theorem yields \eqref{eq:joint-linear} and \eqref{eq:joint-normal}.

\medskip
\noindent \textbf{Proof of Corollary~\ref{cor:profile-location}}. The upper-left block of $\info(\vartheta_0)^{-1}$ is the inverse Schur complement $\info_{\alpha\alpha\cdot\beta}^{-1}$. Equivalently, the efficient score for $\alpha$ is the residual obtained by projecting $\score_\alpha$ onto the linear span of $\score_\beta$, namely \eqref{eq:efficient-score}. Extracting the $\alpha$-coordinates from \eqref{eq:joint-linear} yields \eqref{eq:profile-linear}. The covariance of the sum in \eqref{eq:profile-linear} is $\info_{\alpha\alpha\cdot\beta}$, so the central limit theorem gives \eqref{eq:profile-normal}.

\medskip 
\noindent \textbf{Proof of Theorem~\ref{thm:lan}}. Fix a compact set $C\subset\R^p$. For each $i$ and $h\in C$, the third-order Taylor expansion of $m_{\vartheta_{n,h}}(X_i)$ around $\vartheta_0$ yields
\begin{align*}
m_{\vartheta_{n,h}}(X_i)-m_{\vartheta_0}(X_i)
&=
\frac{1}{\sqrt n}h^\top \score_{\vartheta_0}(X_i)
+\frac{1}{2n} h^\top \partial_\vartheta^2 m_{\vartheta_0}(X_i) h
+R_{i,n}(h),
\end{align*}
where
\[
|R_{i,n}(h)|
\le
\frac{\norm{h}^3}{6n^{3/2}}
\sup_{\vartheta\in U_C}
\bigl\|\partial_\vartheta^3 m_\vartheta(X_i)\bigr\|
\]
for some compact neighborhood $U_C$ of $\vartheta_0$ containing all $\vartheta_{n,h}$ for $n$ large and $h\in C$.

Summing over $i$ gives
\begin{equation}
\label{eq:lan-intermediate}
\Lambda_n(h)
=
h^\top \Delta_n
+
\frac12 h^\top
\Bigl(\frac1n\sum_{i=1}^n \partial_\vartheta^2 m_{\vartheta_0}(X_i)\Bigr)h
+
\sum_{i=1}^n R_{i,n}(h).
\end{equation}
By Proposition~\ref{prop:envelope}, the third derivative has an integrable envelope on $U_C$, so
\[
\sup_{h\in C}\Bigl|\sum_{i=1}^n R_{i,n}(h)\Bigr|
\le
\frac{\sup_{h\in C}\norm{h}^3}{6\sqrt n}\cdot
\frac1n\sum_{i=1}^n
\sup_{\vartheta\in U_C}\bigl\|\partial_\vartheta^3 m_\vartheta(X_i)\bigr\|
\to 0
\qquad\text{a.s.}
\]
Similarly, by Lemma~\ref{lem:ulln},
\[
\frac1n\sum_{i=1}^n \partial_\vartheta^2 m_{\vartheta_0}(X_i)
\to
\E_{\vartheta_0}\bigl[\partial_\vartheta^2 m_{\vartheta_0}(X)\bigr]
=
-\info(\vartheta_0)
\qquad\text{a.s.}
\]
by Lemma~\ref{lem:scoremean}. Substituting these two facts into \eqref{eq:lan-intermediate} yields \eqref{eq:lan}, uniformly on $C$. Finally, $\Delta_n\Rightarrow N(0,\info(\vartheta_0))$ follows from the multivariate CLT, as in the proof of Theorem~\ref{thm:joint-asym-normal}.

\medskip 
\noindent \textbf{Proof of Lemma~\ref{lem:loss-equivalence}}. Let
\[
D(u,v):=\rho\bigl(\Exp_{\mu_0}(u),\Exp_{\mu_0}(v)\bigr)^2 .
\]
The squared distance function is smooth in a neighborhood of $(0,0)$ in
$T_{\mu_0}\Hh^d\times T_{\mu_0}\Hh^d$, because hyperbolic space has no cut locus. Its Taylor expansion at $(0,0)$ is
\[
D(u,v)=\|u-v\|^2+O\bigl((\|u\|+\|v\|)^2\|u-v\|^2\bigr),
\]
uniformly for $(u,v)$ in a sufficiently small fixed ball; this is the standard normal-coordinate expansion of squared Riemannian distance, using
$g_{ij}(w)=\delta_{ij}+O(\|w\|^2)$ near the origin of the chart. If $\sqrt n u_n$ and $\sqrt n v_n$ range over a fixed compact set, then $\|u_n\|+\|v_n\|=O(n^{-1/2})$ and $\|u_n-v_n\|=O(n^{-1/2})$. Therefore
\[
D(u_n,v_n)-\|u_n-v_n\|^2=O(n^{-2})
\]
uniformly on that compact set. Multiplying by $n$ gives the stated convergence.

\medskip 
\noindent \textbf{Proof of Theorem~\ref{thm:lower-bound}}. For fixed $A<\infty$, define the bounded continuous loss
$
w_A(z)=A\wedge \|z\|^2$ for $z \in \bbR^d$. By Theorem~\ref{thm:lan}, the local experiments are LAN with information matrix $\info(\vartheta_0)$. Applying the \HajekLecam local asymptotic minimax theorem to the $d$-dimensional parameter of interest $h_\alpha$, with nuisance $h_\beta$, gives
\begin{equation}
\label{eq:euclidean-hlc-lower}
\lim_{c\to\infty}\liminf_{n\to\infty}\sup_{\|h\|\le c}
\E_{n,h}\bigl[w_A(\tilde a_n-h_\alpha)\bigr]
\ge
\E w_A(Z),
\end{equation}
where $Z\sim N(0,\info_{\alpha\alpha\cdot\beta}^{-1})$; see \citet{hajek_1972_LocalAsymptoticMinimax} and \citet[Ch.~9]{vandervaart_1998_AsymptoticStatistics}.

Now use \eqref{eq:exp-nondecreasing} with
$u=\Log_{\mu_0}(\tilde\mu_n)$ and $v=h_\alpha/\sqrt n$. It gives
\[
n\rho(\tilde\mu_n,\mu_{n,h})^2
\ge
\|\sqrt n\Log_{\mu_0}(\tilde\mu_n)-h_\alpha\|^2
=
\|\tilde a_n-h_\alpha\|^2.
\]
Therefore
\[
A\wedge n\rho(\tilde\mu_n,\mu_{n,h})^2
\ge
w_A(\tilde a_n-h_\alpha).
\]
Combining this inequality with \eqref{eq:euclidean-hlc-lower} proves \eqref{eq:minimax-lower-trunc}. Since the untruncated risk dominates the truncated risk for every $A$, and since
\[
\lim_{A\to\infty}\E[A\wedge\|Z\|^2]=\E\|Z\|^2=\tr(\info_{\alpha\alpha\cdot\beta}^{-1}),
\]
monotone convergence yields \eqref{eq:minimax-lower}.

\medskip
\noindent \textbf{Proof of Theorem~\ref{thm:lam-profile}}. We start by recalling the local regularity fact needed for risk convergence. For every fixed $c<\infty$,
\begin{equation}
\label{eq:uniform-local-linear}
\sqrt n(\hat\vartheta_n-\vartheta_{n,h})
=
\info(\vartheta_0)^{-1}
\frac1{\sqrt n}\sum_{i=1}^n \score_{\vartheta_{n,h}}(X_i)+o_{P_{n,h}}(1),
\end{equation}
uniformly over $\|h\|\le c$. This is obtained by repeating the proof of Theorem~\ref{thm:joint-asym-normal} with the Taylor expansion centered at the local truth $\vartheta_{n,h}$ rather than at $\vartheta_0$. The derivative envelopes in Proposition~\ref{prop:envelope} are valid on one compact neighborhood containing all $\vartheta_{n,h}$ with $\|h\|\le c$ and all large $n$. The compact-uniform law of large numbers in Lemma~\ref{lem:ulln} therefore gives the Hessian convergence uniformly over this local set, and $\info(\vartheta_{n,h})\to\info(\vartheta_0)$ uniformly over $\|h\|\le c$. The Lindeberg condition for the triangular score arrays follows from Proposition~\ref{prop:envelope} together with the uniform moment bound in Lemma~\ref{lem:moments}, applied on the compact local parameter set. Consequently the location component satisfies, uniformly over $\|h\|\le c$,
\begin{equation}
\label{eq:regular-location}
\sqrt n\,\hat\alpha_n-h_\alpha
\Rightarrow
Z\sim N(0,\info_{\alpha\alpha\cdot\beta}^{-1})
\qquad\text{under }P_{n,h}.
\end{equation}
Equivalently, expectations of bounded Lipschitz functions of $\sqrt n\hat\alpha_n-h_\alpha$ converge uniformly on every compact local parameter set, following the usual regularity conclusion for an efficient MLE in a LAN family \citep[Secs.~8.4 and 9.8]{vandervaart_1998_AsymptoticStatistics}.

For bounded continuous losses, \eqref{eq:regular-location} gives convergence of risks. To apply this to the truncated geodesic loss, combine the tightness implied by \eqref{eq:regular-location} with Lemma~\ref{lem:loss-equivalence}. On events where $\sqrt n\hat\alpha_n$ and $h_\alpha$ remain in a fixed compact set,
\[
A\wedge n\rho(\hat\mu_n,\mu_{n,h})^2
-
A\wedge \|\sqrt n\hat\alpha_n-h_\alpha\|^2
\to 0
\]
in probability, uniformly over $\|h\|\le c$. The complement of such compact sets has probability arbitrarily small uniformly over $\|h\|\le c$. Since both truncated losses are bounded by $A$, their expectations have the same limit. Therefore \eqref{eq:lam-attain-trunc} holds.

If Assumption~\ref{ass:ui} holds, the passage from truncated to untruncated squared loss is justified uniformly over $\|h\|\le c$. Letting $A\to\infty$ and using $\E\|Z\|^2=\tr(\info_{\alpha\alpha\cdot\beta}^{-1})$ gives \eqref{eq:lam-attain}. The lower bound in Theorem~\ref{thm:lower-bound} then proves the claimed untruncated local asymptotic minimaxity.

\medskip
\noindent \textbf{Proof of Proposition~\ref{prop:spectral-clipping}}
By orthogonal equivariance of the shell and of the objective, it suffices to treat the case
\[
S=\mathrm{diag}(s_1,\dots,s_d).
\]
Let $B$ be a minimizer. For any skew-symmetric matrix $K$, the path
\[
B_t=e^{tK}B e^{-tK}
\]
remains in $\shell$. Since $B$ minimizes the objective, the derivative at $t=0$ must vanish:
\[
0
=
\frac{d}{dt}\bigg|_{t=0}
\tr\{S B_t^{-1}\}
=
\tr\{[S,B^{-1}]K\}
\qquad
\text{for every skew-symmetric }K .
\]
Because $S$ and $B^{-1}$ are symmetric, the commutator $[S,B^{-1}]$ is skew-symmetric. Taking $
K=[S,B^{-1}]$
gives
$
0
=
\tr\{[S,B^{-1}]^2\}
=
-\|[S,B^{-1}]\|_F^2.
$ Hence $[S,B^{-1}]=0$, and therefore $B$ commutes with $S$. Thus $B$ is block diagonal with respect to the eigenspaces of $S$. Within each eigenspace corresponding to a repeated eigenvalue $s$, the objective contribution is
$
\log\det B_s+s\,\tr(B_s^{-1}),
$
where $B_s$ is the corresponding block. Diagonalizing $B_s$ reduces this block contribution to a sum of scalar terms
\[
\sum_j \left(\log\sigma_j+\frac{s}{\sigma_j}\right),
\qquad
\lambda_-\le \sigma_j\le \lambda_+.
\]
Hence the entire problem reduces to the scalar minimizations
\[
\min_{\lambda_-\le \sigma\le \lambda_+}
\left(\log\sigma+\frac{s_j}{\sigma}\right),
\qquad j=1,\dots,d.
\]
For each $j$,
\[
\frac{d}{d\sigma}
\left(\log\sigma+\frac{s_j}{\sigma}\right)
=
\frac{\sigma-s_j}{\sigma^2}.
\]
The scalar objective is decreasing on $(0,s_j)$ and increasing on $(s_j,\infty)$, so its unique minimizer over $[\lambda_-,\lambda_+]$ is
\[
\widetilde s_j
=
\min\{\lambda_+,\max\{\lambda_-,s_j\}\}.
\]
This proves \eqref{eq:spectral-clipping}. The uniqueness follows from the uniqueness of each scalar minimizer, including within repeated-eigenvalue subspaces.


\end{document}